\documentclass[11pt]{amsart}

\usepackage{latexsym}
\usepackage{amssymb,amsfonts,amsmath,amsthm}
\usepackage{graphicx}
\usepackage{a4wide}

\newcommand{\A}{\mathcal{A}}

\newcommand{\eps}{\varepsilon}

\newcommand{\Bin}{\mathsf{Bin}}

\newcommand{\MP}{\mathsf{MP}}

\newcommand{\E}{\mathcal{E}}
\newcommand{\Prb}[1]{\mathrm{Pr}\Bigl[\,#1\,\Bigr]}

\newcommand{\V}{\mathrm{V}_N}
\newcommand{\D}{\mathcal{D}_R}
\newcommand{\G}{\mathcal{G}}

\newtheorem{theorem}{Theorem}  

\newtheorem{lemma}[theorem]{Lemma}

\newtheorem{claim}[theorem]{Claim}
\newtheorem{corollary}[theorem]{Corollary}

\newtheorem{definition}[theorem]{Definition}

\numberwithin{theorem}{section}
\numberwithin{equation}{section}

\title{On the evolution of random graphs on spaces of negative curvature}

 \author{Nikolaos Fountoulakis}
 \thanks{This research has been supported by a Marie Curie Career Integration Grant PCIG09-GA2011-293619.}
\date{\today}

\begin{document}

\maketitle 
\begin{abstract}
In this work, we study a family of random geometric graphs on hyperbolic spaces. In this setting, $N$ points are chosen randomly 
on a hyperbolic space and any two of them are joined by an edge with probability that depends on their hyperbolic distance, 
independently of every other pair. In particular, when the positions of the points have been fixed, the distribution over 
the set of graphs on these points is the Boltzmann distribution, where the Hamiltonian is given by the sum of  
weighted indicator functions for each pair of points, with the weight being proportional to a real parameter $\beta>0$ (interpreted as the 
inverse temperature) as well as to the hyperbolic distance between the corresponding points. This class of random graphs was introduced by
Krioukov et al.~\cite{ar:Krioukov}. 
We provide a rigorous analysis of aspects of this model and its dependence on the parameter $\beta$. We show that a phase transition 
occurs around $\beta =1$. More specifically, we show that when $\beta > 1$ the degree of a typical vertex is bounded in probability (in fact it
follows a distribution which for large values exhibits a power-law tail whose exponent depends only on the curvature of the space), whereas 
for $\beta <1$ the degree is a random variable whose expected value grows polynomially in $N$. When $\beta = 1$, we establish 
logarithmic growth. 

For the case $\beta > 1$, we establish a connection with a class of inhomogeneous random graphs known as the \emph{Chung-Lu} model. 
Assume that we use the Poincar\'e disc representation of a hyperbolic space. If we condition on the distance of each one of the points from the
origin, then the probability that two given points are adjacent is expressed through the kernel of this inhomogeneous
random graph. 
\end{abstract}

\section{Introduction} 

The theory of geometric random graphs was initiated by Gilbert~\cite{Gilbert61} already in 1961 in the context of 
what is called \emph{continuum percolation}. There, a random infinite graph is formed whose vertex set is the set of 
points of a stationary Poisson point process in the 2-dimensional Euclidean space and two vertices/points are joined 
when their distance is smaller than some certain threshold. The parameter explored there is the probability that a given vertex   
is contained in an infinite component. About a decade later, in 1972, Hafner~\cite{ar:Hafner72} focused on the typical 
properties of large but finite random geometric graphs. Here $N$ points are sampled within a certain region of $\mathbb{R}^d$ 
following a certain distribution (most usually this is the uniform distribution or the distribution of the point-set of a Poisson 
point process) and any two of them are joined when their Euclidean distance is smaller than some threshold which, in general, is 
a function of $N$. In the last two decades, this kind of random graphs was studied extensively by several research groups -- see 
the monograph of Penrose~\cite{bk:Penrose} and the references therein. Numerous typical properties of such random graphs 
have been investigated, such as the chromatic number~\cite{McDiarmid}, Hamiltonicity~\cite{Balogh} etc. 

From the point of view of applications, random geometric graphs on Euclidean spaces have been considered as models for wireless
communication networks. Though the above model might seem slightly simplistic, more complicated random models have been developed
which incorporate various parameters of actual wireless networks; see for example~\cite{Gilbert61} or~\cite{ar:Bac} for a more 
sophisticated model. 

However, what structural characteristics emerge when one considers these points distributed on a curved space where distances are measured
through some (non-Euclidean) metric? Such a model was introduced 
by Krioukov et al.~\cite{ar:Krioukov} and some typical properties of these random graphs were studied with the use of non-rigorous methods.

\subsection{Random geometric graphs on a hyperbolic space} 
The most common representations of the hyperbolic space is the upper-half plane representation $\{z \ : \ \Im z > 0 \}$ as 
well as the Poincar\'e unit disc which is simply the open disc of radius one, that is, $\{(u,v) \in \mathbb{R}^2 \ : \ 1-u^2-v^2 > 0 \}$. 
Both spaces are equipped with the hyperbolic metric; in the former
case this is ${1\over (\zeta y)^2}dy^2$ whereas in the latter this is ${4\over \zeta^2}~{du^2 + dv^2\over (1-u^2-v^2)^2}$, where 
$\zeta$ is some positive real number. 
It can be shown that the (Gaussian) curvature in both cases is equal to $-\zeta^2$ and the two spaces are isometric, that is, there 
exists a bijection between the two spaces which preserves (hyperbolic) distances. In fact, there are more representations of the 2-dimensional
hyperbolic space of curvature $-\zeta^2$ which are isometrically equivalent to the above two. We will denote by $\mathbb{H}^2_\zeta$ the
class of these spaces.

In this paper, following the definitions in~\cite{ar:Krioukov}, we shall be using the native representation of $\mathbb{H}^2_\zeta$. 
Under this representation, the ground space of $\mathbb{H}^2_{\zeta}$ is $\mathbb{R}^2$ and every point $x \in \mathbb{R}^2$ whose
polar coordinates are $(r,\theta)$ has hyperbolic distance from the origin equal to $r$. Also, a circle of radius $r$ around the origin 
has length equal to $2\pi \sinh \zeta r$ and area equal to $2\pi (\cosh \zeta r - 1)$.  

We are now ready to give the definitions of the two basic models introduced in~\cite{ar:Krioukov}.
Consider the native representation of the hyperbolic space of curvature $K = - \zeta^2$, for some $0 < \zeta < 2$. 
Let $N= e^{\zeta R/2}$ -- thus $R$ is a function of $N$ and in particular 
$\zeta R= 2\log N$. 
We create a random graph  by selecting randomly $N$ points from the disc of radius $R$ centred at the origin $O$, which we denote 
by $\D$. The distribution of these points is as follows.  Assume that a random point $u$ has 
polar coordinates $(r, \theta)$. Then $\theta$ is uniformly distributed in $(0,2\pi]$, whereas the probability density function of 
$r$, which we denote by $\rho (r)$, is determined by a parameter $\alpha >0$ and is equal to 
\begin{equation} \label{eq:pdf}
 \rho (r) = \alpha {\sinh  \alpha r \over \cosh \alpha R - 1}. 
\end{equation}
When $\alpha = \zeta$, then this is the uniform distribution.  
This set of points will be the vertex set of the random graph and we will be denoting this random vertex set by $\V$. 
We will be also treating the vertices as points in the hyperbolic space indistinguishably.  
\begin{enumerate}
\item[1.] \emph{The disc model}
This model is the most commonly studied in the theory of random geometric graphs on Euclidean spaces. 
We join two vertices if they are within (hyperbolic) distance $R$ from each other. 

\item[2.] \emph{The binomial model}
We join any two distinct vertices $u, v$ with probability 
$$ p_{u,v} = {1\over \exp\left(\beta ~ {\zeta \over 2}(d(u,v)- R)\right) + 1},$$
independently of every other pair, where $\beta >0$ is fixed and $d(u,v)$ is the hyperbolic distance between $u$ and $v$. 
We denote the resulting random graph by $\G (N;\zeta, \alpha, \beta)$. 
\end{enumerate}
The parameter $\beta >0$ is interpreted as the inverse of the temperature of a fermionic system where particles correspond to 
edges. The distance between two points determines the field that is incurred by the pair.  
In particular, the field that is incurred by the pair $\{u,v \}$ is $\omega_{u,v}=\beta~{\zeta \over 2}~(d(u,v) - R)$. 

An edge between two points corresponds to a particle that ``occupies" the pair. 
In turn, the Hamiltonian of a graph $G$ on the $N$ points, assuming that their positions on $\D$ have been realized, is 
$H(G)= \sum_{u,v} \omega_{u,v} e_{u,v}$, where $e_{u,v}$ is 
the indicator that is equal to $1$ if and only if the edge between $u$ and $v$ is present. (Here the sum is over all distinct unordered pairs of
points.) Each graph $G$ has probability weight that 
is equal to $e^{-H(G)}/Z$, where $Z= \prod_{u,v} \left( 1 + e^{-\omega_{u,v}}\right)$ is the normalizing factor also  known as 
the partition function. 
It can be shown (cf. ~\cite{ar:Park} for example) that in this distribution the probability that $u$ is adjacent to $v$ is equal to 
$1/({e^{\omega_{u,v}}+1})$. See also~\cite{ar:Krioukov} for a more detailed description.  

In this paper, we focus on this model. As we shall see, the central structural features of the resulting random graph heavily 
depend on the value of $\beta$. In particular, we shall distinguish between three regimes: 
\begin{enumerate}
\item[1.] $\beta >1$ (\emph{cold regime}) the random graph $\G (N;\zeta, \alpha, \beta)$ has 
constant average degree depending on $\beta$ and $\zeta$.
\item[2.]  $\beta =1$ (\emph{critical regime}) the average degree grows logarithmically in $N$. 
\item[3.] $\beta < 1$ (\emph{hot regime}) the average degree of $\G (N;\zeta, \alpha, \beta)$ grows polynomially in $N$. 
\end{enumerate}

We shall make these findings precise immediately. 

\subsection{Results} 

The main results of this paper describe the degree of a typical vertex of $\G (N;\zeta, \alpha, \beta)$. 
For a vertex $u \in V_N$ we let $D_u$ denote the degree of $u$ in $\G(N;\zeta, \alpha, \beta)$. 
Our first theorem regards the cold regime. 
\begin{theorem} \label{thm:DegreeDistribution} 
If $\beta > 1$ and $0 < \zeta/ \alpha < 2$, then 
$$ D_u \stackrel{d}{\rightarrow} \MP(F), $$
where $\MP (F)$ denotes a random variable that follows the mixed Poisson distribution with mixing distribution $F$ such that   
$$F(t) = 1 - \left({K \over t}\right)^{2\alpha/\zeta},$$
for any $t\geq K$, where $K = {4\alpha \over 2\alpha - \zeta}~{1 \over \beta}~\sin^{-1} \left( {\pi \over \beta} \right)$ and
 $F(t)= 0$ otherwise. 

In particular, as $k$ and $N$ grow
$$ \Prb{D_u = k} = {2\alpha\over \zeta}~K^{2\alpha /\zeta}~k^{-(2\alpha /\zeta +1)} +o(1), $$
that is, the degree of vertex $u$ follows a power law with exponent $2\alpha /\zeta +1$. 
\end{theorem}
We also show a law of large numbers for the fraction of vertices of any given degree in $\G (N;\zeta, \alpha, \beta)$.
\begin{theorem} \label{thm:Concentration} 
For any $k$ let $N_k$ denote the number of vertices of degree $k$ in $\G (N;\zeta, \alpha, \beta)$. Let $\beta > 1$ and 
$0 < \zeta/\alpha < 2$. For any fixed integer $k\geq 0$, we have 
$$ {N_k \over N} \stackrel{p}\rightarrow \Prb {\MP(F) = k},$$
as $N \rightarrow \infty$.
\end{theorem}
The distribution of the degree of a given vertex changes abruptly when $\beta \leq 1$. 
For a vertex $v \in \V$, if $r_v$ is the distance $v$ from the origin, then we set $t_v = R - r_v$ -- 
we call this quantity the \emph{type} of vertex $v$. We will show that a.a.s. all vertices have their types less than $\zeta R/(2\alpha) $ 
(cf. Corollary~\ref{cor:EffectiveArea}).  
When $\beta = 1$, the degree of any vertex conditional on its type is binomially distributed with expected value proportional 
to $R$. Let $\Bin (n,p)$ denote the binomial distribution with parameters $n,p$. 
For a random variable $X$ we write $X\stackrel{\Delta}{=} (1\pm \eps) \Bin (n,p)$ to denote the fact that the distribution of $X$ is stochastically 
between two random variables distributed as $\Bin (n, (1- \eps)p)$ and $\Bin (n, (1+\eps)p)$, respectively. 
\begin{theorem} \label{thm:CriticalRegime}
Let $\beta = 1$ and $0 < \zeta/\alpha < 2$. For any vertex $u \in \V$, if $t_u \leq \zeta R/(2\alpha)$, then for any $\eps >0$
$$ D_u \stackrel{\Delta}{=} (1\pm \eps) \Bin \left(N-1, (R-t_u)~{Ke^{\zeta t_u /2}\over N} \right),$$
where $K={1\over \pi}~{2\alpha \zeta \over 2\alpha -\zeta}$.
\end{theorem}
When $0 < \beta <1$, the expected degree grows polynomially in $N$. More precisely, the following holds. 
\begin{theorem} \label{thm:HotRegime} 
Let $0 < \beta < 1$ and $0 < \zeta/\alpha < 2$. For any vertex $u \in \V$, if $t_u \leq \zeta R/(2\alpha)$, then for any $\eps >0$
$$ D_u \stackrel{\Delta}{=} (1\pm \eps) \Bin \left(N-1, K\left( {e^{\zeta t_u /2}\over N}\right)^\beta \right),$$
where $K={1\over \sqrt{\pi}}~{2\alpha \over 2\alpha -\beta \zeta}~{\Gamma \left({1-\beta \over 2} \right) \over 
\Gamma\left(1 -{\beta \over 2} \right)}$.
\end{theorem}

\subsection{The disc model vs the binomial model} 

The disc model can be viewed as a limiting case of the binomial model as $\beta \rightarrow \infty$. Assume that the positions 
of the vertices in $\D$ have been realized. If $u,v \in V_N$ are such that $d(u,v) < R$, then when $\beta \rightarrow \infty$ the 
probability that $u$ and $v$ are adjacent tends to 1; however, if $d(u,v)>R$, the this probability converges to 0 as $\beta$ grows. 
In other words, the disc model is the ``frozen" version ($T=0$) of the binomial model. 
Rigorous results for the disc model were obtained by Gugelmann et al.~\cite{ar:Gugel}, regarding their degree sequence as well as 
the clustering coefficient.  

\subsection{Hyperbolic random graphs as a model for social networks} 

We will now discuss the potential of using random graphs on spaces of negative curvature as a model for real-world networks. 
The typical features of real-world networks can be summarised as follows: 
\begin{enumerate}
\item[1.] they are \emph{large}, that is, they contain thousands or millions of nodes;
\item[2.] they are \emph{sparse}, that is, the number of their edges is proportional to the number of nodes; 
\item[3.] they exhibit the \emph{small world phenomenon}; most pairs of vertices are within a short distance from each other;
\item[4.] a significant amount of \emph{clustering} is present. The latter means that two nodes of the network that have a common 
neighbour are somewhat more likely to be connected with each other;  
\item[5.] their degree sequence follows a \emph{power-law} distribution. 
\end{enumerate}
See the book of Chung and Lu~\cite{ChungLuBook+} for a more detailed exposition of these properties. 

Over the last decade a number of models have been developed whose aim was to capture these features. 
Among the first such models is the \emph{preferential attachment model}. 
This is a class of models of randomly growing graphs whose aim is to capture a 
basic feature of real-world networks: nodes which are already popular tend to become more popular as the network grows. 
These models were first defined by Albert and Barab\'asi~\cite{BarAlb} and subsequently defined and studied rigorously 
by Bollob\'as and Riordan (see for example~\cite{DegSeq},~\cite{Diam}).

Another extensively studied model was defined by Chung and Lu~\cite{ChungLu1+}, \cite{ChungLuComp+}. 
Here every vertex has a weight which effectively corresponds to its expected degree and every two vertices are joined independently 
of every other pair with probability that is proportional to the product of their weights. 
When the weights are set suitably, then the resulting random graph has power-law degree distribution. 

All these models are nonetheless insufficient in the 
sense that none of them succeeds in incorporating all the above features. 
For example the Chung and Lu model although it exhibits a power law 
degree distribution and average distance of order $O(\log \log N)$, when the exponent of the power law is between 2 and 3 
(see~\cite{ChungLu1+}) (with $N$ being the number of nodes of the random network) it is locally tree-like around a typical vertex. 
Thus, for the majority of the vertices their neighbourhoods form an independent set. This is also the situation
regarding the Barab\'asi-Albert model. 
Thus, it seems that there is a ``missing link" in the definitions of these models which is a key ingredient to the process of creating 
a social network. 
It seems plausible that the factor which is missing in these models is the \emph{hierarchical structure} of a social network. 
To be more precise, the hierarchies are not among nodes, but more importantly on the level of groups of nodes. 

Real-world networks consist of heterogeneous nodes, which can be classified into groups. In turn, these groups
can be classified into larger groups which consist of smaller subgroups and so on. For example, if we consider the network of citations, 
whose set of nodes is the set of research papers and there is a link from one paper to another if one cites the other, there is a 
natural classification of the nodes according to the scientific fields each paper belongs to (see for example~\cite{Boerner04+}). In the
case of the network of web pages, a similar classification can be considered in terms of the similarity between two web pages. 
That is, the more similar two web pages are, the more likely it is that there exists a hyperlink between them (see~\cite{Mencezer02+}).  

This classification can be approximated by tree-like structures representing the hidden hierarchy of the network. 
The tree-likeness suggests that in fact the geometry of this hierarchy is \emph{hyperbolic}. As we have already seen in the above 
definitions, the volume growth in the hyperbolic space is exponential which is also the case, for example, 
when one considers a $k$-ary tree, that is, a rooted tree where every vertex has $k$ children. 
Let us consider the Poincar\'e disc model. If we place the root of an infinite $k$-ary tree at the centre of 
the disc, then the hyperbolic metric provides the necessary room to embed the tree into the disc so that every edge has unit length
in the embedding. See also~\cite{ar:Krioukov} for a related discussion.

In this contribution, however, we have only explored the degree sequence of random graphs on a hyperbolic space showing that 
in the sparse regime, they exhibit a power-law degree distribution with exponent that depends on the curvature of the underlying space. 
It remains to show that they also exhibit the small world phenomenon, proving that the typical distance between two randomly 
chosen vertices is small, for example $O(\log \log N)$.
Regarding the presence of clustering, Krioukov et al.~\cite{ar:Krioukov} have shown with the use of non-rigorous arguments that these 
random graphs do exhibit clustering which can be adjusted through the parameter $\beta$. This dependence also remains to be quantified 
rigorously. For the disc model, this existence of clustering has been verified by Gugelmann et al.~\cite{ar:Gugel}.   

\subsection{Outline} 
We begin our analysis with some preliminary results which will be used throughout our proofs. In Section~\ref{sec:distribution}, we prove 
Theorem~\ref{thm:DegreeDistribution} and we continue in Section~\ref{sec:Corr} with the analysis of the asymptotic correlation 
of the degrees of finite collections of vertices and the proof of Theorem~\ref{thm:Correlations}. Its proof immediately yields 
Theorem~\ref{thm:Concentration} with the use of Chebyschev's inequality. 

\section{Preliminaries} \label{sec:prelim}
  
The next lemma shows that the type of a vertex is essentially exponentially distributed. 
\begin{lemma} \label{lem:TypeDistr}
For any $v \in \V$ we have 
$$ \Prb { t_v \leq x} = 1 - e^{-\alpha x} + O \left( N^{-2\alpha /\zeta} \right),$$
uniformly for all $0\leq x \leq R$.
\end{lemma}
\begin{proof} 
We use the definition of $\rho (r)$ and write 
\begin{equation*} 
\begin{split} 
\Prb { t_v \leq x} &= \alpha \int_{R-x}^R {\sinh (\alpha r) \over \cosh (\alpha R) - 1} dr = {1\over \cosh (\alpha R) - 1}\left[ \cosh 
(\alpha R) - \cosh (\alpha (R-x)) \right]. 
\end{split}
\end{equation*}
Now, note that $\cosh (\alpha R) ={1\over 2}~e^{\alpha R}(1+o(1))= {1\over 2}~N^{2\alpha /\zeta}(1+o(1))$ and therefore 
$${\cosh (\alpha R) \over  \cosh (\alpha R) -1 } = 1 +O(N^{-2\alpha /\zeta}).$$ 
Also, 
\begin{equation*}
\begin{split}
{\cosh (\alpha (R-x)) \over \cosh (\alpha R) - 1} &= 
{\cosh (\alpha (R-x))\over \cosh (\alpha R)}\left( 1+ O(N^{-2\alpha/\zeta}) \right) \\
&= {e^{\alpha (R-x)} \over e^{\alpha R}}~
{1+ e^{-2\alpha (R-x)}\over 1+ e^{-2\alpha R}}~\left( 1+ O(N^{-2\alpha /\zeta}) \right) \\
&= e^{-\alpha x} \left( 1+ e^{-2\alpha (R-x)} \right)~\left( 1 - e^{-2\alpha R} + O(e^{-4\alpha R}) \right)~
 \left( 1+ O(N^{-2\alpha/\zeta}) \right).
\end{split}
\end{equation*}
But $e^{-x -2(R-x)}=e^{-2R +x}$ and $x \leq R$. Thus $e^{-\alpha x -2\alpha (R-x)} \leq e^{-\alpha R} = N^{-2\alpha /\zeta}$,
which implies the statement of the lemma.
\end{proof}
Now, let us set $x_0 =  \zeta R/(2\alpha)  + \omega (N)$, where $\omega (N) \rightarrow \infty$ as $N \rightarrow \infty$ is an arbitrarily slowly 
growing function.
The above lemma immediately yields the following corollary. 
\begin{corollary} \label{cor:EffectiveArea}
If $0< \zeta/\alpha < 2$, then a.a.s all vertices $v \in \V$ have $t_v \leq x_0$.
\end{corollary}  
\begin{proof}
Note that $x_0 = {1\over \alpha}~\log N + \omega (N)$. 
For a vertex $v \in \V$ applying Lemma~\ref{lem:TypeDistr}, we have 
$$ \Prb {t_v > x_0} = e^{-x_0} + O \left( N^{-2\alpha /\zeta} \right) = o\left( {1\over N}\right),$$
since $\zeta / \alpha < 2$.  The corollary follows from Markov's inequality.
\end{proof}
We will also need an estimate on the distance between two points in the case their relative angle is not too small (this is the 
typical case). 
\begin{lemma} \label{lem:Dist}
Assume that $0 < \zeta / \alpha < 2$.
Let $u, v$ be two distinct points in $\D$ such that $t_u, t_v \leq x_0$ and let $\theta_{u,v}$ denote their relative radius. 
Let also $\hat{\theta}_{u,v}:= \left( e^{-2\zeta (R-t_u)} + e^{-2\zeta (R-t_v)}\right)^{1/2}$.
If  
$$\hat{\theta}_{u,v} \ll \theta_{u,v} \leq  \pi ,$$
then 
$$ d(u,v) = 2R - (t_u + t_v) + {2\over \zeta}~\log \sin \left( {\theta_{u,v} \over 2} \right) + 
O \left( \left( {\hat{\theta}_{u,v} \over \theta_{u,v}} \right)^2 \right),$$
uniformly for all $u, v$ with $t_u, t_v \leq x_0$.  
\end{lemma}
\begin{proof} 
We begin with the hyperbolic law of cosines: 
\begin{equation*} \label{eq:CosinesLawInit}
\cosh (\zeta d(u,v)) = \cosh (\zeta (R- t_u)) \cosh (\zeta (R-t_v)) - \sinh (\zeta (R- t_u)) \sinh (\zeta (R-t_v)) \cos ( \theta_{u,v} ).
\end{equation*}
Since $t_u, t_v \leq x_0$, it follows that both $R-t_u, R-t_v \rightarrow \infty$ as $N \rightarrow \infty$. Thus the right-hand 
side of (\ref{eq:CosinesLawInit}) becomes:
\begin{equation*} 
\begin{split} 
&\cosh (\zeta (R- t_u)) \cosh (\zeta (R-t_v)) - \sinh (\zeta (R- t_u)) \sinh (\zeta (R-t_v)) \cos ( \theta_{u,v} ) = \\
& {e^{\zeta (2R- (t_u + t_v))} \over 4} \left( \left(1+e^{-2\zeta (R-t_u)}\right)\left(1+e^{-2\zeta (R-t_v)}\right) 
-\left(1-e^{-2\zeta (R-t_u)}\right)\left(1-e^{-2\zeta (R-t_v)}\right) \cos (\theta_{u,v}) \right) \\
& = {e^{\zeta (2R- (t_u + t_v))} \over 4} \left(1- \cos (\theta_{u,v}) + \left(1+ \cos (\theta_{u,v}) \right) 
\left( e^{-2\zeta (R-t_u)} +  e^{-2\zeta (R-t_v)}\right) + O\left( e^{-2\zeta (2R- (t_u + t_v))}\right)\right). 
\end{split}
\end{equation*}
By the convexity of the function $e^{-2\zeta x}$, we have 
\begin{equation} \label{eq:convexity}
e^{-\zeta (2R- (t_u + t_v))} = e^{-2\zeta {2R- (t_u + t_v)\over 2}} \leq {1\over 2} \left( e^{-2\zeta (R-t_u)} +  e^{-2\zeta (R-t_v)}\right) 
\leq \hat{\theta}_{u,v}^2. 
\end{equation} 
Thus, the previous estimate can be written as 
\begin{equation*}
\cosh (\zeta d(u,v)) = {e^{\zeta (2R- (t_u + t_v))} \over 4} \left(1- \cos (\theta_{u,v}) + \left(1+ \cos (\theta_{u,v}) \right) 
\hat{\theta}_{u,v}^2 + O\left( \hat{\theta}_{u,v}^4 \right)\right). 
\end{equation*}
If $\theta_{u,v}$ is bounded away from 0, then clearly $1-\cos(\theta_{u,v})$ dominates the expression in brackets. 
Now assume that $\theta_{u,v} = o(1)$.
It is a basic trigonometric identity that $1- \cos (\theta_{u,v}) = 2\sin^2 \left( {\theta_{u,v} \over 2} \right)$.  
Then $1- \cos (\theta_{u,v}) = {\theta_{u,v}^2 \over 2} (1-o(1))$. But the assumption that $\theta_{u,v} \gg \hat{\theta}_{u,v}$ 
again implies that also in this case $1-\cos(\theta_{u,v})$ dominates expression in brackets. Thus
\begin{equation} \label{eq:CosinesLaw} 
\begin{split} 
\cosh (\zeta d(u,v)) &= 
{e^{\zeta (2R- (t_u + t_v))} \over 4}~\left(1- \cos (\theta_{u,v})\right) ~\left(1 + 
O \left( \left( {\hat{\theta}_{u,v} \over \theta_{u,v}} \right)^2 \right) \right) \\
&= {e^{\zeta (2R- (t_u + t_v))} \over 2}~\sin^2 \left( {\theta_{u,v} \over 2} \right)~\left(1 + 
O \left( \left( {\hat{\theta}_{u,v} \over \theta_{u,v}} \right)^2 \right) \right).
\end{split}
\end{equation}
Now, we take logarithms in (\ref{eq:CosinesLaw}) and divide both sides by $\zeta$ thus obtaining: 
\begin{equation} \label{eq:d(u,v)}
d(u,v) + {1 \over \zeta} \log \left( 1-e^{-2\zeta d(u,v)} \right) = 2R -t_u - t_v + {2\over \zeta} \log \sin \left( {\theta_{u,v} \over 2} \right)
+ O \left( \left( {\hat{\theta}_{u,v} \over \theta_{u,v}} \right)^2 \right).
\end{equation}
We now need to give an asymptotic estimate on $e^{-2\zeta d(u,v)}$. We derive this from (\ref{eq:CosinesLaw}) as well. 
For $N$ large enough, we have 
\begin{equation*} 
e^{\zeta d(u,v)} \geq {e^{\zeta (2R- (t_u + t_v))} \over 4}~\sin^2 \left( {\theta_{u,v} \over 2} \right) \geq 
{1 \over 32}~e^{\zeta (2R- (t_u + t_v))}~\theta_{u,v}^2 \stackrel{(\ref{eq:convexity})}{\geq}  {1\over 32}~ \left( {\theta_{u,v} \over 
\hat{\theta}_{u,v}} \right)^2,
\end{equation*}
from which it follows that 
$$  \left| \log \left( 1-e^{-2\zeta d(u,v)} \right) \right|  = O \left( \left( {\hat{\theta}_{u,v} \over \theta_{u,v}} \right)^4 \right).$$
Substituting this into (\ref{eq:d(u,v)}) completes the proof of the lemma.
\end{proof}
Let $\hat{p}_{u,v} = {1\over \pi}~\int_{0}^\pi p_{u,v} d \theta$ - this is the probability that two points $u$ and $v$ are connected by 
an edge, conditional on their types. For an arbitrary slowly growing function 
$\omega: \mathbb{N} \rightarrow \mathbb{N}$
we define $$\mathcal{D}_{R,\omega}^{(2)}= \{ u,v \in \D \ : t_v,t_u \leq x_0,\ R-t_u-t_v \geq \omega (N) \}.$$ 
Also, for two points $u, v \in \D$ we set  $A(t_u,t_v) = \exp \left({\zeta \over 2} \left(R- (t_u + t_v) \right) \right)$. 
The following lemma gives an asymptotic estimate on $\hat{p}_{u,v}$, for all values of $\beta$, in terms of $A_{t_u,t_v}$. 
\begin{lemma} \label{lem:AngleAv}
Let $\beta >0$. There exists a constant $C_{\beta}>0$ such that uniformly for all $u,v \in \mathcal{D}_{R,\omega}^{(2)}$
we have  
$$ \hat{p}_{uv} = \begin{cases}
(1+o(1)){C_{\beta}\over A_{u,v}}~, & \mbox{if $\beta>1 $} \\
& \\
(1+o(1)){C_\beta \ln A_{u,v} \over A_{u,v}},  & \mbox{if $\beta = 1 $} \\
& \\
(1+o(1)){C_\beta \over A_{u,v}^{\beta}}, & \mbox{if $\beta < 1 $}
 \end{cases}.$$
In particular, 
$$ C_\beta = \begin{cases} {2\over \beta}~\sin^{-1} \left( {\pi \over \beta}\right), & \mbox{if $\beta >1 $} \\
{2 \over \pi}, & \mbox{if $\beta = 1 $} \\
{1 \over \sqrt{\pi}}~{ \Gamma \left({1-\beta \over 2} \right) \over \Gamma\left(1 -{\beta \over 2} \right)}, & \mbox{if $\beta < 1$}
\end{cases}.$$ 
\end{lemma}
\begin{proof} 
Throughout this proof we write $A_{u,v}$ for $A(t_u,t_v)$. 
Recall that 
$$ p_{u,v} = {1\over \exp\left(\beta ~ {\zeta \over 2}(d(u,v)- R)\right) + 1}. $$
We will estimate the integral of $p_{u,v}$ over $\theta_{u,v}$. When $\theta_{u,v}$ is within the range given 
in Lemma~\ref{lem:Dist} we will use the estimate given there.  
In particular, we shall define $\tilde{\theta}_{u,v}\gg \hat{\theta}_{u,v}$ and split the integral into two 
parts, namely when $0\leq \theta_{u,v} < \tilde{\theta}_{u,v}$ and when $\tilde{\theta}_{u,v} \leq \theta_{u,v} \leq \pi$.
The parameter $\tilde{\theta}_{u,v}$ is close to $A_{u,v}^{-1}$. 
Note that the following holds.
\begin{claim} \label{clm:IntBndrs}
If $R-t_u-t_v \rightarrow \infty$ as $N \rightarrow \infty$, then 
$$ A_{u,v}^{-1} \gg \hat{\theta}_{u,v}.$$
\end{claim}
We postpone the proof of this claim until later. 

Let $\omega (N)$ be slowly enough growing so that in the following definition of $\tilde{\theta}_{u,v}$, we have 
$\tilde{\theta}_{u,v} \gg \hat{\theta}_{u,v}$, when $\beta \geq 1$, and $\tilde{\theta}_{u,v} = o(A_{u,v}^{-\beta})$, when $\beta < 1$.
We set 
$$\tilde{\theta}_{u,v} = \begin{cases}{A_{u,v}^{-1}\over \omega (N)}, & \ \mbox{if $\beta \geq 1$} \\ 
\omega (N) A_{u,v}^{-1}, & \ \mbox{if $\beta < 1$}
\end{cases}.$$ 
Thus when $\tilde{\theta}_{u,v} \leq \theta_{u,v} \leq \pi$, we use Lemma~\ref{lem:Dist} and write 
\begin{equation}  \label{eq:DenomApx} 
\begin{split}
\exp\left(\beta ~ {\zeta \over 2}(d(u,v)- R)\right) = C~e^{\beta {\zeta\over 2} (R - (t_u + t_v)) + \beta~  
\log \sin (\theta_{u,v}/2)}, 
\end{split}
\end{equation}
where $C = 1 + O \left( \left( {\hat{\theta}_{u,v} \over \tilde{\theta}_{u,v}}\right)^2 \right)$. 
By Claim~\ref{clm:IntBndrs} and the choice of the function $\omega(N)$, we have that $C=1+o(1)$.

We decompose the integral that gives $\hat{p}_{u,v}$ into two parts which we bound separately. 
 \begin{equation} \label{eq:InitSplit}
\hat{p}_{u,v} = {1\over \pi}~\int_{0}^{\pi} p_{u,v} d\theta ={1 \over \pi}~\int_{0}^{\tilde{\theta}_{u,v}} p_{u,v} d \theta 
+ {1 \over \pi}~\int_{\tilde{\theta}_{u,v}}^{\pi} p_{u,v} d \theta. 
\end{equation}
The first integral can bounded trivially as follows:
\begin{equation}\label{eq:1stIntegral}
\int_{0}^{\tilde{\theta}_{u,v}} p_{u,v} d \theta  \leq \tilde{\theta}_{u,v} = \begin{cases} 
o \left( A_{u,v}^{-1} \right), & \ \mbox{if $\beta \geq 1$} \\
o \left( A_{u,v}^{-\beta} \right), & \ \mbox{if $\beta < 1$}
\end{cases}.
\end{equation}
We now focus on the second integral in (\ref{eq:InitSplit}). 
We will treat the cases $\beta < 1$ and $\beta \geq 1$ separately, starting with the former one. 

\medskip 

\noindent
$\beta < 1$ 
\smallskip

\noindent
Recall that $\tilde{\theta}_{u,v}$ is such that $\tilde{\theta}_{u,v} \gg
A_{u,v}^{-1}$. 
Thus, we write 
\begin{equation} \label{eq:HotLowerBound}  
\begin{split}
&\int_{\tilde{\theta}_{u,v}}^{\pi} p_{u,v} d \theta   
= \int_{\tilde{\theta}_{u,v}}^{\pi} {1 \over C A_{u,v}^{\beta} \sin^{\beta} \left({\theta \over 2}\right)+1} d\theta 
= {1+o(1) \over A_{u,v}^{\beta}}~\int_{\tilde{\theta}_{u,v}}^{\pi} {1 \over \sin^{\beta} \left({\theta \over 2}\right) + \Theta 
\left( A_{u,v}^{-\beta}\right) } d \theta \\ 
&= 
{1+o(1) \over  A_{u,v}^{\beta}}~\int_{\tilde{\theta}_{u,v}}^{\pi} {1 \over \sin^{\beta} \left({\theta \over 2}\right)} d \theta 
= {1+o(1) \over A_{u,v}^{\beta}}~\int_{0}^{\pi} {1 \over \sin^{\beta} \left({\theta \over 2}\right)} d \theta. 
\end{split}
\end{equation}
Substituting the estimates of (\ref{eq:1stIntegral}) and (\ref{eq:HotLowerBound}) into (\ref{eq:InitSplit}) we obtain
\begin{equation} \label{eq:HotLemma}
\hat{p}_{u,v}  = (1+o(1))~{1 \over \pi}~\left( \int_{0}^{\pi} {1 \over \sin^{\beta} \left({\theta \over 2}\right)} d \theta \right)~ 
{1\over A_{u,v}^{\beta}}, 
\end{equation}
uniformly for all $u,v \in \mathcal{D}_{R,\omega}^{(2)}$.
Finally, note that 
$$ \int_{0}^{\pi} {1 \over \sin^{\beta} \left({\theta \over 2}\right)} d \theta = 2  
\int_{0}^{\pi/2} {1 \over \sin^{\beta} \left( \theta \right)} d \theta = \sqrt{\pi}~{\Gamma \left({1-\beta \over 2} \right) \over \Gamma\left(1 -{\beta \over 2} \right)}. $$
\medskip 

\noindent
$\beta \geq 1$ 
\smallskip

\noindent
We use the inequality $\sin \theta \leq \theta$, which holds for all $\theta \in [0, \pi]$, and obtain an upper bound on the right-hand side of 
(\ref{eq:DenomApx}). 
\begin{equation*} 
\begin{split} 
\exp\left(\beta ~ {\zeta \over 2}(d(u,v)- R)\right)  
\leq C~e^{\beta {\zeta\over 2} (R - (t_u + t_v))}~ \left( {\theta_{u,v} \over 2} \right)^{\beta}  = C~A_{u,v}^{\beta}
~\left( {\theta_{u,v} \over 2} \right)^\beta.
\end{split} 
\end{equation*}
Using this bound we can bound the second integral in (\ref{eq:InitSplit}) from below as follows.  
\begin{equation} \label{eq:LowerBound}
\begin{split}
\int_{\tilde{\theta}_{u,v}}^{\pi} p_{u,v} d \theta \geq \int_{\tilde{\theta}_{u,v}}^{\pi } 
{1\over C~A_{u,v}^{\beta}~\left( {\theta \over 2} \right)^\beta +1} d \theta.
\end{split}
\end{equation}
We perform a change of variable setting $z= C^{1/\beta}~A_{u,v}~{\theta \over 2}$. Thus with $C' = C^{1/\beta}/2$ 
the integral on the right-hand side of (\ref{eq:LowerBound}) becomes
\begin{equation} \label{eq:IntLower}
\begin{split}
&\int_{\tilde{\theta}_{u,v}}^{\pi} 
{1\over C~A_{u,v}^{\beta}~\left( {\theta \over 2} \right)^\beta +1} d \theta = {1 \over C'}~{1\over A_{u,v}}~
\int_{C'A_{u,v} \tilde{\theta}_{u,v}}^{C'\pi A_{u,v}} {1\over z^{\beta} + 1} dz. 
\end{split}
\end{equation}
We now provide an estimate for the integral on the right-hand side of (\ref{eq:IntLower}) for any $\beta \geq 1$ - its proof is elementary and 
we omit it. 
\begin{claim} \label{clm:MainIntegral} 
Let $g_1(N)$ and $g_2 (N)$ be non-negative real-valued functions on the set of natural numbers, such that 
$g_1 (N) \rightarrow 0$ and $g_2 (N) \rightarrow \infty$ as $N \rightarrow \infty$. We have 
$$ 
\int_{g_1(N)}^{g_2 (N)} {1\over z^\beta + 1}dz =
\begin{cases} 
(1+o(1)) \int_{0}^{\infty} {1\over z^{\beta} + 1}dz, &  \mbox{if $\beta > 1$} \\
& \\
(1+o(1)) \ln g_2 (N), & \mbox{if $\beta = 1$} 
\end{cases}.
$$
\end{claim}  
We take $g_1 (N) = C'A_{u,v} \tilde{\theta}_{u,v} = C'/ \omega (N) \rightarrow 0$ and 
$g_2 (N) = C'\pi A_{u,v} \rightarrow \infty$, since $u, v \in \mathcal{D}_{R,\omega}^{(2)}$, and we obtain through (\ref{eq:LowerBound}): 
\begin{equation} \label{eq:LowerFinal}
\int_{\tilde{\theta}_{u,v}}^{\pi} p_{u,v} d \theta \geq 
\begin{cases} 
 \left( 1+o(1)\right)~{2 \over C^{1/\beta}}~{1\over A_{u,v}}~\int_{0}^{\infty} {1\over z^\beta + 1}dz, & \mbox{if $\beta > 1$} \\
 & \\
 \left( 1+ o(1)\right)~{2 \over C^{1/\beta}}~{\ln A_{u,v}\over A_{u,v}}, & \mbox{if $\beta = 1$} 
\end{cases} .
\end{equation}

To deduce the upper bound we will split the integral into two parts. 
For an $\eps \in (0,\pi)$, we write
\begin{equation} \label{eq:IntSplit} 
\int_{\tilde{\theta}_{u,v}}^{\pi} p_{u,v} d \theta = 
 \int_{\tilde{\theta}_{u,v}}^{\eps} p_{u,v} d \theta + \int_{\eps}^{\pi} p_{u,v} d \theta.
\end{equation}
We will bound each one of the integrals on the right-hand side separately. 

For the first integral we will use (\ref{eq:DenomApx}) together with the bound $\sin \theta \geq \theta - \theta^3$, which holds 
for any $\theta \leq \eps$, provided that the latter is sufficiently small. 
More specifically, we let $\eps = \eps (N) \in (0, \pi)$ be a slowly decaying function so that $A_{u,v} \eps (N) \rightarrow \infty$ as 
$N \rightarrow \infty$.
We shall also use $(1-\theta^2)^{\beta} \geq 1-\beta \theta^2$.
Thus (after a change of variable where we replace $\theta/2$ by $\theta$) for sufficiently large $N$, we have 
\begin{equation*}
\begin{split}
 \int_{\tilde{\theta}_{u,v}}^{\eps} p_{u,v} d \theta & \leq 2\int_{\tilde{\theta}_{u,v}/2}^{\eps/2} 
  {1 \over C A_{u,v}^{\beta} (\theta - \theta^3)^{\beta} +1} d \theta \leq  
  2\int_{\tilde{\theta}_{u,v}/2}^{\eps/2} 
  {1 \over C A_{u,v}^{\beta} \theta^\beta (1 - \beta \theta^2 ) +1} d \theta \\
  & \leq 2\int_{\tilde{\theta}_{u,v}/2}^{\eps/2} 
  {1 \over C A_{u,v}^{\beta} \theta^\beta (1 - \beta \eps^2 /4 ) +1} d \theta.
\end{split}
\end{equation*}
We change the variable in the last integral setting $z= \left[C (1 - \beta \eps^2 /4 ) \right]^{1/\beta} A_{u,v} \theta$. 
Thus, we obtain: 
\begin{equation*}
\begin{split} 
\int_{\tilde{\theta}_{u,v}}^{\hat{\theta}_{u,v}'} p_{u,v} d \theta  \leq 
{2\over  \left[C (1 - \beta \eps^2 /4 ) \right]^{1/\beta} A_{u,v}} 
\int_{B_1}^{B_2} {1\over z^{\beta} + 1} dz,
\end{split}
\end{equation*} 
where $B_1 = \left[C (1 - \beta \eps^2 /4 ) \right]^{1/\beta} A_{u,v}\tilde{\theta}_{u,v}/2$ and 
$B_2 = \left[C (1 - \beta \eps^2 /4 ) \right]^{1/\beta} A_{u,v} \eps/2 $.
We have $B_1=o(1)$ whereas, 
$B_2 =\left[C (1 - \beta \eps^2 /4 ) \right]^{1/\beta} A_{u,v} \eps/2 
\rightarrow \infty$. So, Claim~\ref{clm:MainIntegral} yields:
\begin{equation}\label{eq:2ndIntegral}
 \int_{\tilde{\theta}_{u,v}}^{\eps} p_{u,v} d \theta \leq 
\begin{cases}
(1+o(1)) ~{2 \over A_{u,v}}~ \int_{0}^{\infty} {1\over z^{\beta} + 1}dz, & \mbox{if $\beta > 1$} \\
 & \\
(1+o(1))~{2\ln A_{u,v} \over A_{u,v}}, & \mbox{if $\beta = 1$} \\ 
\end{cases}.  
\end{equation}
The second integral in (\ref{eq:IntSplit}) can be bounded easily. 
$$  \int_{\eps}^{\pi} p_{u,v} d \theta = 
\int_{\eps}^{\pi} 
 {1 \over C A_{u,v}^{\beta}~\sin^{\beta}(\theta /2) +1} d \theta = O \left( {1\over \left( A_{u,v} \eps \right)^{\beta}}\right) = 
o(A_{u,v}^{-1}), $$
if $\eps$ is decaying slowly enough. 
Hence, uniformly for all $u, v \in  \mathcal{D}_{R,\omega}^{(2)}$ we have
\begin{equation}\label{eq:ColdUpperBound}
 \int_{\tilde{\theta}_{u,v}}^{\pi} p_{u,v} d \theta \leq 
\begin{cases}
(1+o(1)) ~{2 \over A_{u,v}}~ \int_{0}^{\infty} {1\over z^{\beta} + 1}dz, & \mbox{if $\beta > 1$} \\
 & \\
(1+o(1))~{2\ln A_{u,v} \over A_{u,v}}, & \mbox{if $\beta = 1$} \\ 
\end{cases}.  
\end{equation}
Thereby, (\ref{eq:InitSplit}) together with (\ref{eq:1stIntegral}) and (\ref{eq:LowerFinal}),(\ref{eq:ColdUpperBound}) yield the lemma for 
$\beta \geq 1$. Finally, note that when $\beta > 1$ we have $\int_{0}^{\infty} {1\over z^{\beta} + 1}dz= {\pi \over \beta} 
\sin^{-1} \left( {\pi \over \beta}\right)$.
 

We now conclude the proof of the lemma with the proof of Claim~\ref{clm:IntBndrs}. 
\begin{proof}[Proof of Claim~\ref{clm:IntBndrs}] 
We will show that $A_{u,v}^{-1} \gg \hat{\theta}_{u,v}$. Equivalently, it is sufficient to show that 
\begin{equation} \label{eq:ToProve}
{\zeta \over 2} \left( R-(t_u+t_v) \right) + {1\over 2} \log \left( e^{-2\zeta(R-t_u)} + e^{-2\zeta(R-t_v)}\right) \rightarrow -\infty, 
\end{equation}
as $N \rightarrow \infty$. Adding and subtracting $R$ inside the brackets in the first summand, we obtain: 
\begin{equation*} 
\begin{split}
& {\zeta \over 2} \left( R-(t_u+t_v) \right) + {1\over 2} \log \left( e^{-2\zeta(R-t_u)} + e^{-2\zeta(R-t_v)}\right) \\
&= -{\zeta R \over 2} + {1\over 2} \left(\zeta (R-t_u) +\zeta (R-t_v)  \right) + 
{1\over 2} \log \left( e^{-2\zeta(R-t_u)} + e^{-2\zeta(R-t_v)}\right).
\end{split}
\end{equation*}
For notational convenience, we write $a = \zeta (R-t_u)$ and $b = \zeta (R-t_v)$. 
Without loss of generality, assume that $a \leq b$.
Thus, the above expression is now written as 
\begin{equation*}
\begin{split} 
&-{\zeta R \over 2} + {1\over 2} \left( a +b  \right) +  {1\over 2} \log \left( e^{-2a } + e^{-2b}\right) = 
 -{\zeta R \over 2} + {1\over 2} \left[ a +b   +   \log e^{-2a} + \log(1+e^{-2(b - a)}) \right] \\
& = -{\zeta R \over 2} + {1\over 2} \left[ b  - a + \log(1+e^{-2(b - a)}) \right] \leq
-{\zeta R \over 2} + {1\over 2} \left[ b  - a + e^{-2(b - a)} \right].  
\end{split}
\end{equation*}
Note that $b - a \leq R \left(\zeta -1 + \zeta /(2\alpha) \right) + \omega (N)$. Therefore, 
\begin{equation*}
\begin{split} 
-{\zeta R \over 2} + {1\over 2} \left( a  + b \right) +  {1\over 2} \log \left( e^{-2a} + e^{-2b}\right) 
\leq -R \left(1 - {\zeta \over 2 \alpha} \right) +{ \omega (N) + 1 \over 2}\rightarrow -\infty,
\end{split}
\end{equation*}
if $\omega (N)$ grows slowly enough. 
\end{proof}
\end{proof}
\subsection{Interlude: hyperbolic random graphs as inhomogeneous random graphs}
The notion of \emph{inhomogeneous random graphs}, was introduced S\"oderberg~\cite{ar:s02} and
was studied in great detail by Bollob\'as, Janson and Riordan in~\cite{BJR}.
In its most general setting, there is an underlying compact metric space $\mathcal{S}$ equipped with a measure $\mu$ on its 
Borel $\sigma$-algebra. This is the space of \emph{types} of the vertices. A \emph{kernel} $\kappa$ is a bounded real-valued,
non-negative function on $\mathcal{S} \times \mathcal{S}$, which is symmetric. It is assumed that the vertices of the random 
graph are points in $\mathcal{S}$. If $x, y \in \mathcal{S}$, then the corresponding vertices are joined with probability 
that is equal to ${\kappa (x,y) \over N} \wedge 1$, where $N$ is the total number of vertices, independently of every other pair. 
The points that are the vertices of the graph are approximately distributed according to $\mu$. 
More specifically, the empirical distribution function on the $N$ points converges weakly to $\mu$ as $N \rightarrow \infty$. 

When $\beta > 1$, the above lemma gives an expression for the probability that two vertices $u$ and $v$ having 
types $t_u$ and $t_v$, respectively, are adjacent. This expression is proportional to $A_{u,v}^{-1}= {e^{\zeta t_u/2}e^{\zeta t_v/2} 
\over N}$. Thus a hyperbolic random graph at the cold regime may be viewed as an inhomogeneous random graph 
on $N$ vertices with kernel function which is equal (up to a multiplicative constant) to $(1/x)^{\zeta/2}(1/y)^{\zeta /2}$, where 
$x,y \in (0,1]$. (Here, we have applied the transformation $e^{t_u}=1/T_u$, where $T_u \in (0,1]$.) 
In fact, this kernel corresponds to the \emph{Chung-Lu} model of random graphs with given expected degrees - see 
\cite{ChungLu1+}, \cite{ChungLuComp+} as well as (6.1.20) on page 124 in~\cite{bk:vdH}. 
When $\beta < 1$, the corresponding kernel is that of a not-too-sparse inhomogeneous random graph. 

However, this analogy is not precise, as if we condition on the types of the vertices the edges do not appear independently.
As we shall see later in our analysis (cf. Section~\ref{sec:Corr}), if we condition on the event that $v_1$ and $v_2$ are both adjacent 
to a vertex $u$, then this increases the probability that $v_1$ is adjacent to $v_2$.

\section{The distribution of the degree of a vertex} \label{sec:distribution}
In this section, we prove Theorem~\ref{thm:DegreeDistribution}. 
Let us fix some $u \in V_N$. 
We will condition on the position of $u$ in $\D$; in particular, we will condition on $t_u$ and the angle $\theta_u$. 
In this section as well as later, we will be writing $\rho(t)$ for $\alpha \sinh (\alpha(R-t)) /(\cosh (\alpha R) -1)$. 

Let us fix first $t_u$ and $\theta_u$ such that $t_{u} \leq x_0$ and $\theta_u \in (0,2\pi]$. 
(Recall that by Corollary~\ref{cor:EffectiveArea} we have that a.a.s. $t_u \leq x_0 = \zeta R/(2\alpha) + \omega (N)$ for all $u \in V_N$.) 
Denoting by $V_N^u$ the set $V_N \setminus \{  u\}$, we write $D_u = \sum_{v \in V_N^u} I_{uv}$, where $I_{uv}$ is the indicator random 
variable that is equal to 1 if and only if the edge $\{ u,v \}$ is present in $\G (N;\zeta ,\beta)$. 
Note that conditional on $t_u$ and $\theta_u$ the family $\{ I_{uv}\}_{v \in V_N^u}$ is a family of independent and 
identically distributed random variables. 

We begin with the estimation of the expectation of $I_{uv}$ for an arbitrary 
$v \in V_{N}^u$  conditional on $t_u$ and $\theta_u$. 
\begin{lemma} \label{lem:IndExp} 
Let $\beta > 0$ and $0< \zeta/ \alpha < 2$. There exists a constant $K=K(\zeta, \beta, \alpha )>0$ such that
uniformly for all $t_u \leq x_0$ and $\theta_u \in (0,2 \pi ]$, we have 
$$\Prb {I_{uv} = 1 \ | \ t_u, \theta_u} = 
\begin{cases}
\left(1+ o(1) \right)~K~{e^{\zeta t_u /2} \over N}, & \mbox{if $\beta > 1$} \\
& \\
(1+o(1))~K~(R-t_u)~ {e^{\zeta t_u /2}\over N}, & \mbox{if $\beta = 1$} \\
 & \\
(1+o(1))~K~\left({e^{\zeta t_u /2} \over N} \right)^{\beta}, & \mbox{if $\beta < 1$}
\end{cases}.
$$
In particular, we have 
$$
K= \begin{cases} 
 {4 \alpha \over 2\alpha - \zeta}~{1 \over \beta}~\sin^{-1} \left({\pi \over \beta} \right), & \mbox{if $\beta > 1$} \\
  \\
  {1\over \pi}~{2\alpha \zeta \over 2\alpha - \zeta}, & \mbox{if $\beta = 1$} \\
  \\
  {1\over \sqrt{\pi}}~{2\alpha  \over 2\alpha -\beta \zeta}~{\Gamma \left({1-\beta \over 2} \right) \over 
\Gamma\left(1 -{\beta \over 2} \right)}, & 
\mbox{if $\beta < 1$}\\
\end{cases}.
 $$
\end{lemma} 
\begin{proof}
We write $\hat{p} (t_u , t_v)$ for $\hat{p}_{u,v}$ as the latter depends only on $t_u$ and $t_v$. 
Hence, we have 
\begin{equation} \label{eq:IndExp} 
\begin{split}
\Prb {I_{uv} = 1 \ | \ t_u, \theta_u} &= \int_0^R \hat{p}(t_u,t_v)\rho (t_v) d t_v \\
& = \int_0^{R-t_u - \omega (N)} \hat{p}(t_u,t_v)\rho (t_v) d t_v 
  + \int_{R-t_u - \omega (N)}^R \hat{p}(t_u,t_v)\rho (t_v) d t_v.
\end{split}
\end{equation}
The second integral can be bounded as follows. 
\begin{equation} \label{eq:IndExp2nd}
\begin{split}
&\int_{R-t_u - \omega (N)}^R \hat{p}(t_u,t_v)\rho (t_v) d t_v \leq 
\int_{R-t_u - \omega (N)}^R \rho (t_v) d t_v = {1\over \cosh (\alpha R) -1} \left[ \cosh (\alpha (t_u + \omega (N))) - \cosh (0) \right] \\
& = {e^{\alpha (t_u + \omega (N))} \over \cosh (\alpha R) - 1} (1-o(1)) =  {e^{\alpha (t_u + \omega (N))} \over N^{2\alpha /\zeta}} (1-o(1)) = 
e^{\alpha \omega (N)}\left( {e^{\zeta t_u /2}\over N} \right)^{2\alpha /\zeta} (1-o(1)). 
\end{split}
\end{equation}
For the first integral we use the estimates obtained in Lemma~\ref{lem:AngleAv}.  
We set $K_1:=(1+o(1)) C_\beta$. We treat each one of the three cases separately. 

\medskip
\noindent
$\beta > 1$
\smallskip

\noindent
In this case,  we have: 
\begin{equation} \label{eq:ColdIndExp1st} 
\begin{split}
 &\int_0^{R-t_u - \omega (N)} \hat{p}(t_u,t_v)\rho (t_v) d t_v = K_1  \int_0^{R-t_u - \omega (N)} {1\over A (t_u,t_v)}~\rho (t_v) d t_v \\
&= K_1 e^{\zeta t_u /2}  \int_0^{R-t_u - \omega (N)} e^{-{\zeta \over 2}(R-t_v)}~\rho (t_v) d t_v  \\
&= \alpha K_1 e^{\zeta t_u /2}  \int_0^{R-t_u - \omega (N)} e^{-{\zeta \over 2}(R-t_v)}
~{\sinh (\alpha (R-t_v))\over \cosh (\alpha R) -1}d t_v  \\
&= \alpha K_1 {e^{\zeta t_u /2} \over \cosh (\alpha R) -1} \int_{t_u + \omega (N)}^{R}  e^{-{\zeta t_v\over 2}}~\sinh (\alpha t_v) d t_v \\
&=(1+o(1))~{\alpha K_1\over 2}~{e^{\zeta t_u /2} \over \cosh (\alpha R) -1}~\int_{t_u + \omega (N)}^{R} e^{t_v 
\left( \alpha- {\zeta \over 2} \right)} dt_v \\ 
&= (1+o(1))~{\alpha K_1 \over 2 \alpha - \zeta}~{e^{\zeta t_u /2} \over \cosh (\alpha R) -1}
~\left[ e^{R \left(\alpha - {\zeta \over 2} \right)} - e^{\left(\alpha - {\zeta \over 2}\right)(t_u + \omega(N))} \right] \\
& = (1+o(1))~{2\alpha K_1 \over 2\alpha  -\zeta}~{e^{\zeta t_u /2} \over e^{\zeta R/2}}~{e^{\alpha R} \over 2( \cosh (\alpha R) -1)} 
=(1+o(1))~K~{e^{\zeta t_u /2} \over N}.
\end{split}
\end{equation}

\medskip 
\noindent
$\beta = 1$ 
\smallskip

\noindent 
Here we perform a similar but somewhat more involved calculation.  
\begin{equation} \label{eq:CritIndExp1st1} 
\begin{split}
 &\int_0^{R-t_u - \omega (N)} \hat{p}(t_u,t_v)\rho (t_v) d t_v = K_1  \int_0^{R-t_u - \omega (N)} 
{\ln A(t_u,t_v) \over A (t_u,t_v)}~\rho (t_v) d t_v \\
&= K_1 e^{\zeta t_u /2}  \int_0^{R-t_u - \omega (N)} \left({\zeta \over 2}(R-t_u-t_v) \right) 
e^{-{\zeta \over 2}(R-t_v)}~\rho (t_v) d t_v  = \\
&=K_1 {\zeta e^{\zeta t_u /2} \over 2} \times \\
&\left[ \int_0^{R-t_u - \omega (N)} \left( R-t_v \right) 
e^{-{\zeta \over 2}(R-t_v)}~\rho (t_v) d t_v  
-   t_u\int_0^{R-t_u - \omega (N)}  e^{-{\zeta \over 2}(R-t_v)}~\rho (t_v) d t_v   \right]. 
\end{split}
\end{equation}
The second integral is as in (\ref{eq:ColdIndExp1st}) 
$$ \int_0^{R-t_u - \omega (N)}  e^{-{\zeta \over 2}(R-t_v)}~\rho (t_v) d t_v  = (1+o(1)) {2\alpha  \over 2\alpha -\zeta}~{1 \over N}.
$$
For the first one, we use the identity $\int_a^b x e^{tx} = {1 \over t} \left( (b-1/t)e^{tb} - (a-1/t)e^{ta} \right)$. We have 
\begin{equation*}
\begin{split}
&\int_0^{R-t_u - \omega (N)} \left( R-t_v \right) e^{-{\zeta \over 2}(R-t_v)}~\rho (t_v) d t_v  = \\
&{\alpha \over \cosh (\alpha R) -1}\int_0^{R-t_u - \omega (N)} \left( R-t_v \right) e^{-{\zeta \over 2}(R-t_v)}~\sinh (\alpha (R-t_v)) d t_v   \\
&=(1-o(1)) {\alpha \over 2(\cosh (\alpha R) -1)} 
\int_0^{R-t_u - \omega (N)} \left( R-t_v \right) e^{\left(\alpha -{\zeta \over 2}\right) (R-t_v)} d t_v  \\
&= (1-o(1)) {\alpha \over 2(\cosh (\alpha R) -1)} \int_{t_u + \omega (N)}^{R} t_v e^{\left(\alpha - {\zeta \over 2} \right) t_v} dt_v \\
&={1- o(1) \over 2\alpha - \zeta}~ {\alpha \over \cosh (\alpha R) -1}\times \\
& \left( \left( R-{1\over \alpha - \zeta /2} \right)
e^{\left(\alpha - {\zeta \over 2} \right) R} 
- \left( t_u + \omega (N) - {1 \over \alpha - \zeta /2} \right) e^{\left(\alpha - {\zeta \over 2} \right) (t_u + \omega(N))}\right).
\end{split}
\end{equation*}
Since $t_u \leq x_0$, we have 
$$ {(t_u + \omega (N) - 1)  e^{\left(\alpha - {\zeta \over 2} \right) (t_u + \omega(N))}\over \cosh (\alpha R) -1} = 
O \left( { N^{1- \zeta /(2\alpha )} \log^2 N\over N^{2\alpha /\zeta}} \right). $$ 
As ${\zeta \over 2\alpha } + {2\alpha \over \zeta} > 2$, the latter is $o\left(R/N \right)$.
Also, 
$$ {Re^{\left(\alpha - {\zeta \over 2} \right) R}  \over \cosh (\alpha R) -1} = (1+o(1))~{2R \over N}.$$
Hence, 
\begin{equation*}
\begin{split} 
\int_0^{R-t_u - \omega (N)} \left( R-t_v \right) e^{-{\zeta \over 2}(R-t_v)}~\rho (t_v) d t_v  = 
(1+o(1)) {2\alpha  \over 2\alpha  - \zeta}~ {R \over N}.
\end{split}
\end{equation*} 
Substituting this bound into (\ref{eq:CritIndExp1st1}), we finally obtain
\begin{equation}\label{eq:CritIndExp1st}
\begin{split} 
\int_0^{R-t_u - \omega (N)} \hat{p}(t_u,t_v)\rho (t_v) d t_v = 
(1+o(1)) K_1 {\alpha \zeta \over 2\alpha -\zeta}~(R-t_u)~{e^{\zeta t_u /2} \over N}.
\end{split}
\end{equation}

\medskip 
\noindent
$\beta < 1$
\smallskip 

\noindent
This case is similar to (\ref{eq:ColdIndExp1st})
\begin{equation} \label{eq:HotIndExp1st} 
\begin{split}
 &\int_0^{R-t_u - \omega (N)} \hat{p}(t_u,t_v)\rho (t_v) d t_v = K_1  \int_0^{R-t_u - \omega (N)} {1\over A (t_u,t_v)^{\beta}}~\rho (t_v) d t_v \\
&= K_1 e^{\beta \zeta t_u /2}  \int_0^{R-t_u - \omega (N)} e^{-\beta{\zeta \over 2}(R-t_v)}~\rho (t_v) d t_v \\
& = \alpha K_1 e^{\beta\zeta t_u /2}  \int_0^{R-t_u - \omega (N)} 
e^{-{\beta\zeta \over 2}(R-t_v)}~{\sinh (\alpha (R-t_v))\over \cosh (\alpha R) -1}d t_v  \\
&=\alpha K_1 {e^{\beta\zeta t_u /2} \over \cosh (\alpha R) -1} \int_{t_u + \omega (N)}^{R}  e^{-{\beta\zeta t_v\over 2}}~\sinh (\alpha
 t_v) d t_v \\ 
&= (1+o(1))~{\alpha K_1\over 2}~{e^{\beta\zeta t_u /2} \over \cosh (\alpha R) -1}~\int_{t_u + \omega (N)}^{R} 
e^{t_v \left(\alpha - \beta{\zeta \over 2} \right)} dt_v \\ 
&= (1+o(1))~{\alpha K_1 \over 2\alpha  -\beta\zeta}~{e^{\beta \zeta t_u /2} 
\over \cosh (\alpha R) -1}~\left[ e^{R \left(\alpha - \beta{\zeta \over 2} \right)} - 
e^{\left(\alpha -\beta{\zeta \over
2}\right)(t_u + \omega(N))} \right] \\
& = (1+o(1))~{2\alpha K_1 \over 2\alpha -\beta\zeta}~{e^{\beta\zeta t_u /2} \over e^{\beta\zeta R/2}}~
{e^{\alpha R} \over 2( \cosh (\alpha R) -1)} 
=(1+o(1))~K~\left({e^{\zeta t_u /2} \over N} \right)^{\beta}.
\end{split}
\end{equation}

\noindent
Thus combining (\ref{eq:IndExp2nd}) together (\ref{eq:ColdIndExp1st}), (\ref{eq:CritIndExp1st}) and (\ref{eq:HotIndExp1st}) 
the lemma follows. 
\end{proof}
It is clear that the proof of the above lemma yields also the following corollary. 
\begin{corollary} \label{cor:IndExpAng}
Let $0 < \zeta / \alpha < 2$.
There exists a constant $K=K(\zeta, \alpha, \beta)>0$ such that uniformly for all $t_u \leq x_0$, we have 
$$\Prb {I_{uv} = 1 \ | \ t_u } = \begin{cases}
\left(1+ o(1) \right)~K~{e^{\zeta t_u /2} \over N}, & \mbox{if $\beta > 1$} \\
& \\
(1+o(1))~K~(R-t_u)~{e^{\zeta t_u /2}\over N}, & \mbox{if $\beta = 1$} \\
 & \\
(1+o(1))~K~\left({e^{\zeta t_u /2} \over N} \right)^{\beta}, & \mbox{if $\beta < 1$}
\end{cases}.
$$
\end{corollary}
Now, Theorems~\ref{thm:CriticalRegime} and~\ref{thm:HotRegime} follow immediately from Corollary~\ref{cor:IndExpAng}. 
For the cold regime we have to work slightly more. 

\subsection{The cold regime $\beta > 1$}
Let $\hat{D}_u$ be a Poisson random variable with parameter equal to 
$T_u := \sum_{v \in V_N^u} \Prb {I_{uv}=1 \ | \ t_u} = (1+o(1))Ke^{\zeta t_u/2}$. 
With $d_{TV}(\cdot, \cdot)$ denoting the total variation distance between two random variables, we deduce the following lemma. 
\begin{lemma} \label{lem:PoissonApx}
Conditional on the value of $t_u$, we have 
$$ d_{TV} \left( D_u, \hat{D}_u \right) =o(1),$$
uniformly for all $t_u \leq R/2 - \omega (N)$.
\end{lemma}
\begin{proof}
Recall that $D_u = \sum_{v \in V_N^u} I_{uv}$ and conditional on $t_u$ and $\theta_u$ this is a sum of independent and identically 
distributed indicator random variables. Keeping $t_u$ fixed this is also the case when we average over $\theta_u \in (0,2\pi ]$.
Thus conditioning only on $t_u$ the family $\{ I_{uv} \}_{v \in V_N^u}$ is still a family of i.i.d. indicator random variables, whose 
expected values are given by Lemma~\ref{lem:IndExp}. 

By Theorem~2.9 in~\cite{bk:vdH} and Corollary~\ref{cor:IndExpAng} we have for $N$ sufficiently large 
$$ d_{TV} (D_u, \hat{D}_u ) \leq \sum_{v \in V_N^u} \Prb {I_{uv}=1 \ | \ t_u}^2 \leq 2 N \left({e^{\zeta R/4 - \zeta \omega (N)/2} \over 
N}\right)^2 = 2 N e^{-\zeta \omega (N)}~ {1\over N} = o(1).$$
\end{proof}
For any integer $k\geq 0$ we have 
\begin{equation} \label{eq:InitConditioning}
\begin{split}
& \Prb {D_u = k}  = {1\over 2\pi} \int_{0}^R \int_{0}^{2\pi} \Prb { D_u = k \ | \ t_u, \theta_u } \rho ( t_u ) d t_u d \theta_u \\
&= {1\over 2\pi } \int_{t_u \leq R/2 - \omega (N)}  \int_{0}^{2\pi} \Prb {D_u = k \ | \ t_u, \theta_u} \rho( t_u) d t_u d \theta_u \\
&+ {1\over 2\pi } \int_{t_u >R/2 - \omega (N)}  \int_{0}^{2\pi} \Prb {D_u = k \ | \ t_u, \theta_u} \rho(t_u) d t_u d \theta_u.
\end{split}
\end{equation}
We bound the second integral as follows. 
\begin{equation} \label{eq:2ndInt}
\begin{split} 
\int_{R/2 - \omega (N) < t_u} \Prb {D_u = k \ | \ t_u} \rho( t_u) d t_u \leq 
\int_{R/2 - \omega (N) < t_u} \rho( t_u) d t_u = o(1).
\end{split}
\end{equation}
We will use Lemma~\ref{lem:PoissonApx} to approximate the first integral. 
\begin{equation} \label{eq:1stInt1stSub}
\begin{split}
&\int_{t_u \leq R/2 - \omega (N)} \Prb {D_u = k \ | \ t_u} \rho( t_u) d t_u = \int_{t_u \leq R/2 - \omega (N)} \Prb {\hat{D}_u = k} 
\rho( t_u) d t_u + o(1) \\
& = \int_{t_u \leq R} \Prb {\hat{D}_u = k} \rho( t_u) d t_u + o(1).
\end{split}
\end{equation}
But recall that $\Prb {\hat{D}_u = k} = \Prb {\mathsf{Po}\left( T_u \right) = k}$. 
Let $K_N = (1+o(1))K$ denote the factor of $e^{\zeta t_u/2}$ in the expression of $T_u$. 
If $t < K$, then $\Prb {T_u \leq t} \rightarrow 0$ as $N\rightarrow \infty$. 
However, for any $t\geq K$ we have 
\begin{equation*}
\begin{split} 
&\Prb {T_u \leq t} = \Prb { t_u \leq {2\over \zeta}~ \ln {t\over K_N}} = 
\alpha \int_{0}^{{2\over \zeta}~ \ln {t\over K_N}} {\sinh (\alpha (R-x)) \over \cosh (\alpha R) -1} dx \\ 
&= \alpha \int_{R-{2\over \zeta}~ \ln {t\over K_N}}^R {\sinh (\alpha x) \over \cosh (\alpha R) - 1} dx \\
&= {1\over \cosh (\alpha R) -1} \left[ \cosh (\alpha R) - \cosh \left( \alpha R-{2\alpha \over \zeta}~ \ln {t\over K_N}\right) \right] 
= 1 - \left( {K \over t} \right)^{2\alpha \over \zeta} + o(1). 
\end{split}
\end{equation*}
In other words, 
$$ \Prb {T_u \leq t} \rightarrow F(t), \ \mbox{as $N \rightarrow \infty$}.$$
Thus, 
$$ \int_{t_u \leq R/2 - \omega (N)} \Prb {D_u = k \ | \ t_u} \rho( t_u) d t_u  \rightarrow \Prb {\MP (F) = k}, \ 
\mbox{as}\ N \rightarrow \infty.$$
The above together with (\ref{eq:2ndInt}) and (\ref{eq:InitConditioning}) complete the proof of Theorem~\ref{thm:DegreeDistribution}.

\subsubsection{Power laws}

We close this section with a simple calculation proving that $\Prb {\MP (F) = k}$ has power-law behaviour with exponent $2\alpha /\zeta +1$ 
as $k$ grows. 
\begin{lemma}
We have
$$ \Prb {\MP (F) = k} \rightarrow {2\alpha \over \zeta} K^{2\alpha /\zeta}~ k^{-2\alpha /\zeta -1} \ \mbox{as} \ k, N \rightarrow \infty.$$
\end{lemma}
\begin{proof}
The pdf of $F(t)$ for $t > K$ is equal to ${2\alpha \over \zeta}~K^{2\alpha /\zeta}/t^{-2\alpha /\zeta - 1}$ and equal to 0 otherwise. 
Thus we have 
\begin{equation} \label{eq:PrbMP}
\begin{split}
\Prb {\MP (F) = k} &={2\alpha \over \zeta}~K^{2\alpha /\zeta}~ \int_K^\infty e^{-t}~{t^k \over k!}~t^{-2\alpha /\zeta -1} dt = 
{2\alpha \over \zeta}~K^{2\alpha /\zeta}~{1\over k!}~ \int_K^\infty e^{-t}~t^{k-2\alpha /\zeta - 1} dt \\
&= {2\alpha \over \zeta}~K^{2\alpha /\zeta}~{1\over k!}~\left( \Gamma (k-2\alpha /\zeta ) - \int_0^{K} 
e^{-t}~t^{k-2\alpha /\zeta - 1} dt \right).
\end{split}
\end{equation}
Note now that the last integral is $O(K^k)$ and therefore, as $k\rightarrow \infty$, we have 
$$ {\int_0^{K} e^{-t}~t^{k+2\alpha /\zeta +1} dt \over k!} = O \left( {K^k \over k!} \right). $$
Now using the standard asymptotics for the Gamma function we have 
$$ \Gamma (k-2\alpha /\zeta ) = (1+o(1)) \sqrt{2\pi (k - 2\alpha /\zeta - 1)}~e^{-k + 2\alpha /\zeta + 1}~\left( k-2\alpha /\zeta -1 \right)^{k-2\alpha /\zeta -1}, $$
and also $k!= (1+o(1))\sqrt{2\pi k}e^{-k}k^k$.
Thus, 
\begin{equation*} \begin{split}{\Gamma (k-2\alpha /\zeta) \over k!} &= (1+o(1)) e^{2\alpha /\zeta + 1} \left( 1 - {2\alpha /\zeta + 1\over
 k} \right)^{k - 2\alpha /\zeta  - 1}~k^{-(2\alpha /\zeta +1)} \\
&=  (1+o(1))~k^{-(2\alpha /\zeta +1)}.
\end{split}
\end{equation*}
Thus, (\ref{eq:PrbMP}) now yields
$$\Prb {\MP (F) = k} = {2\alpha \over \zeta}~K^{2\alpha /\zeta}~k^{-(2\alpha /\zeta +1)}(1 +o(1)).$$
\end{proof}

\section{Asymptotic correlations of degrees} \label{sec:Corr}
In this section, we deal with the correlations of the degrees in the cold regime. We show that the degrees of any finite collection 
of vertices are asymptotically independent.  
\begin{theorem} \label{thm:Correlations} 
Let $\beta > 1$ and $0 < \zeta /\alpha < 2$.
For any integer $m \geq 2$ and for any collection of $m$ vertices $v_1, \ldots, v_m$ their degrees $D_{v_1},\ldots, D_{v_m}$ are
asymptotically independent. 
\end{theorem}
The proof of the above result together with Chebyschev's inequality yield the concentration of the number of vertices of any fixed degree and 
complete the proof of Theorem~\ref{thm:Concentration}. 
  
We proceed with the proof of Theorem~\ref{thm:Correlations}. 
Let us fix $m\geq 2$ distinct vertices $v_1,\ldots, v_m$. Let $k_1,\ldots, k_m$ be non-negative integers.  We will show that 
\begin{equation} \label{eq:2ndThmToProve}
\Prb {D_{v_1}=k_1,\ldots, D_{v_m} = k_m} \rightarrow \Prb {D_{v_1}=k_1} \cdots \Prb {D_{v_m} = k_m}, \ \mbox{as $N \rightarrow \infty$},
\end{equation}
where the convergence is uniform over all choices of the $m$ vertices. 
The proof of Lemma~\ref{lem:AngleAv} suggests that there exists a specific region around each vertex such that if another vertex 
is located outside it, then the probability that the two vertices are joined becomes much less than the estimate given in 
Lemma~\ref{lem:AngleAv}. In other words, this region is where a vertex is most likely to have its neighbours in -- 
see Figure~\ref{fig:vital}. 
\begin{definition} 
For a vertex $u \in V_N$ such that $t_u \leq x_0$, we let $A_u$ be the set of 
points $\{ w \in \D \ : \ t_w \leq x_0, \theta_{uw} \leq \min \{ \pi, \hat{\theta}_{u,w}' \} \}$, where $\hat{\theta}_{u,w}' = \omega(N) A^{-1}_{u,w}$. 
We call this the \emph{vital area} of vertex $u$. 
\end{definition}
\begin{figure}[htp] 
\begin{center}
\includegraphics[scale=0.6]{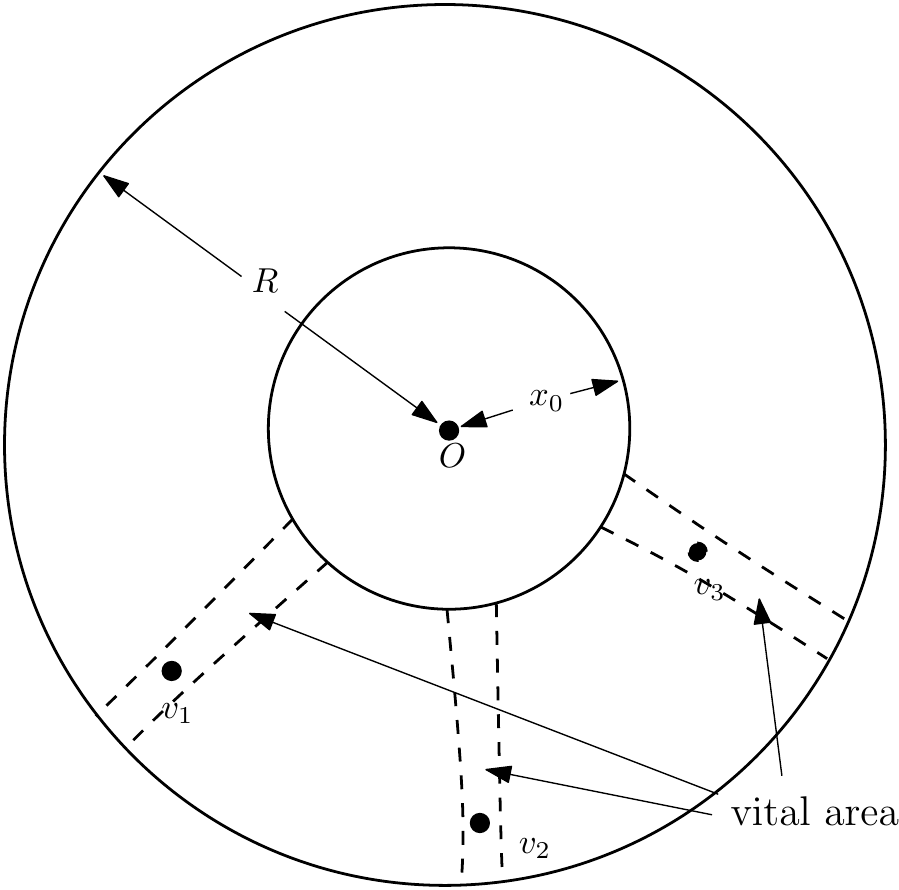}
\caption{The vital area of vertices in $\D$}
\label{fig:vital}
\end{center}
\end{figure}

We begin our analysis proving that the vital areas $A_{u_i}$, for $i=1,\ldots, m$, are mutually disjoint with high probability. 
We let $\E_1$ be this event. Though $m$ is meant to be fixed, the following claim is also valid for $m$ growing as a function of $N$.
\begin{claim} \label{clm:DisjointAreas} If $m \leq {\sqrt{N^{1-\zeta /(2\alpha )}} \over \omega (N) e^{\zeta \omega (N)/2}}$, then 
$\Prb {\E_1} = 1-o(1)$. 
\end{claim}
\begin{proof} 
Let $i \in \{ 1,\ldots, m\}$ and assume that $t_{v_i} < R (1 - \zeta /(2\alpha)) - 2\omega (N)$. 
Considering all vertices $w \in A_{v_i}$ the parameter $\hat{\theta}_{v_i,w}'$ is maximised when 
$t_w = x_0$. So let $\hat{\theta}_{v_i}'$ be this maximum, that is, 
$$\hat{\theta}_{v_i}'= {\omega (N) \over N} \exp \left({\zeta \over 2\alpha } \left(
t_{v_i} + x_0 \right) \right) = {\omega (N) \over N}~e^{\zeta t_{v_i}/2} N^{\zeta /(2\alpha )} e^{\zeta \omega (N)/2} 
= {\omega (N) \over N^{1-\zeta /(2\alpha )}}~ e^{\zeta \omega (N)/2}~e^{\zeta t_{v_i}/2} . $$ 
Observe that our assumption on $t_{v_i}$ implies that 
$$ e^{\zeta t_{v_i}/2} < e^{{\zeta R\over 2}~\left( 1 - {\zeta \over 2 \alpha}  \right) - \zeta \omega (N)} = N^{1-\zeta /(2\alpha )}
 e^{-\zeta \omega (N)}, $$ 
whereby $\hat{\theta}_{v_i}' = o(1)$. 
Thus for $i \not = j$, with $t_{v_i}, t_{v_j} < R (1 - \zeta /(2\alpha)) - 2\omega (N)$ we have $A_{v_i} \cap A_{v_j} \not = \emptyset$ if 
$$ \theta_{v_i v_j} \leq \hat{\theta}_{v_i}'+ \hat{\theta}_{v_j}' = {\omega (N) \over N^{1-\zeta /(2\alpha )}}~e^{\zeta \omega (N)/2}
\left( e^{\zeta t_{v_i}/2} + e^{\zeta t_{v_j}/2}  \right).$$

The probability that this occurs for two given distinct indices $i, j$ is crudely bounded for $N$ large enough as follows:  
\begin{equation*}
\begin{split}
&{1\over \pi}~{\omega (N) \over N^{1-\zeta /(2\alpha )}}~e^{\zeta \omega (N)/2}
\int_0^{x_0} \int_0^{x_0} \left( e^{\zeta t_{v_i}/2} + e^{\zeta t_{v_j}/2}  \right) \rho (t_{v_i})  \rho ( t_{v_j}) dt_{v_i} dt_{v_j} \\
&\leq {2\alpha^2 \over \pi}~{\omega (N) \over N^{1-\zeta /(2\alpha )}}~e^{\zeta \omega (N)/2}
\int_0^{x_0} \int_0^{x_0} \left( e^{\zeta t_{v_i}/2} + e^{\zeta t_{v_j}/2}  \right) e^{-\alpha t_{v_i}} e^{-\alpha t_{v_j}} dt_{v_i} dt_{v_j} \\
&\leq {4\alpha^2 \over \pi}~{\omega (N) \over N^{1-\zeta /(2\alpha )}}~e^{\zeta \omega (N)/2} 
\int_{0}^{x_0} e^{- \left( \alpha  - {\zeta \over 2} \right) x} dx 
\leq {4\alpha^2 \over \pi}~{2 \over 2\alpha  - \zeta}~{\omega (N) \over N^{1-\zeta /2}}~e^{\zeta \omega (N)/2}. 
\end{split}
\end{equation*}
Also, by Lemma~\ref{lem:TypeDistr} for a vertex $v$ we have 
\begin{equation*}
\begin{split}
\Prb { t_{v} \geq R\left( 1 - {\zeta \over 2 \alpha}  \right) - 2 \omega (N)} &= e^{-\alpha R \left( 1 - {\zeta \over 2 \alpha}  \right) + 2\alpha \omega (N) } + O \left( N^{-2\alpha /\zeta} \right) \\ 
& = N^{- {2\alpha \over \zeta}\left( 1 - {\zeta \over 2 \alpha}  \right)}e^{2\alpha \omega (N)} +  O \left( N^{-2\alpha /\zeta} \right) \\ 
&= N^{1 - {2\alpha \over \zeta}} e^{2\alpha \omega (N)} + O \left( N^{-2\alpha /\zeta} \right) .
\end{split}
\end{equation*}
Assume now that $m \leq {\sqrt{N^{1-\zeta /(2\alpha )}} \over \omega (N) e^{\zeta \omega (N)/2}}$. 
Thus the probability that there exists a pair of distinct vertices $v_i, v_j$ with $i,j = 1,\ldots, m$ such that 
$A_{v_i} \cap A_{v_j} \not = \emptyset$ is bounded by  
$$m^2 O \left( {\omega (N) \over N^{1-\zeta /(2\alpha )}}~e^{\zeta \omega (N)/2} \right) 
+ m O \left( N^{1 - {2\alpha \over \zeta}} e^{2\alpha \omega (N)} \right)=o(1).$$ 
\end{proof}
We assume that $m\geq 2$ is fixed and we condition on the event that $t_{v_i} \leq x_0$ for all $i = 1,\ldots ,m$ (which we denote by
$\mathcal{T}_1$) as well as on the event $\E_1$. By Corollary~\ref{cor:EffectiveArea} and Claim~\ref{clm:DisjointAreas} both events occur 
with probability $1-o(1)$. 

For a vertex $w \not \in \{v_1,\ldots, v_m\}$ we denote by $\A^w_{v_i}$ the event that $w$ is located within $A_{v_i}$ and it 
is adjacent to $v_i$. 
In what follows, we drop the superscript $w$ as the probability of $\A_{v_i}^w$ is the same for all $w$. 

Now, let us consider the event that $k_i$ vertices satisfy the event $\A_{v_i}$, for $i=1,\ldots, m$, whereas all other 
vertices do not. We denote this event by $\A (k_1,\ldots, k_m)$. 
Also, for every $i=1,\ldots, m$ let $\tilde{\A}_{v_i}^w$ be the event that a certain vertex $w$ is located outside $A_{v_i}$ and 
is adjacent to $v_i$.  We let $\mathcal{B}_1$ be the event $\cup_{w \in V_{N}^{v_1,\ldots , v_m}} \cup_{i=1}^m \tilde{\A}_{v_i}^w$,
that is, the event that there exists a vertex $w \in [N] \setminus \{v_1,\ldots , v_m \}$ which is adjacent to $v_i$, for some $i=1,\ldots, m$, 
but it is located outside $A_{v_i}$. 
Thus conditional on $\E_1 \cap \mathcal{T}_1$, 
if the event $\mathcal{B}_1$ is \emph{not} realized, then the event that vertex $v_i$ has degree $k_i$, for all 
$i=1,\ldots, m$ is realized if and only if $\A (k_1,\ldots, k_m)$ is realized. 
Using the union bound, we will show that 
\begin{equation} \label{eq:B1} \Prb{\mathcal{B}_1} = o(1).\end{equation} 
(We will show this without any conditioning.)
Thereby, we can deduce the following:
\begin{equation} \label{eq:ToDeduce}
\begin{split}
&\Prb {D_{v_1}=k_1,\ldots, D_{v_m} = k_m} = \Prb {D_{v_1}=k_1,\ldots, D_{v_m} = k_m \ | \ \E_1 \cap \mathcal{T}_1} +o(1) \\
 & \stackrel{(\ref{eq:B1})}{=} 
\Prb {D_{v_1}=k_1,\ldots, D_{v_m} = k_m \ | \ \E_1, \mathcal{T}_1, \overline{\mathcal{B}_1}} + o(1) \\
 &= \Prb {A (k_1,\ldots, k_m)  \ | \ \E_1, \mathcal{T}_1, \overline{\mathcal{B}_1}} + o(1) \\
 &= \Prb {A (k_1,\ldots, k_m), \overline{\mathcal{B}_1} \ | \ \E_1, \mathcal{T}_1} / \Prb {\overline{\mathcal{B}_1} \ | \ \E_1, \mathcal{T}_1} + o(1) \\
&=  \Prb {A (k_1,\ldots, k_m), \overline{\mathcal{B}_1} \ | \ \E_1, \mathcal{T}_1} + o(1) \\
& = \Prb {A (k_1,\ldots, k_m) \ | \ \E_1, \mathcal{T}_1} - \Prb {A (k_1,\ldots, k_m), \mathcal{B}_1 \ | \ \E_1, \mathcal{T}_1} + o(1) \\
& = \Prb {A (k_1,\ldots, k_m) \ | \ \E_1, \mathcal{T}_1} +o(1).
\end{split}
\end{equation}
We will show further that $\Prb {A (k_1,\ldots, k_m) \ | \ \E_1 \cap \mathcal{T}_1}$ is asymptotically equal to the product of the probabilities that $D_{v_i}=k_i$, over 
$i=1,\ldots, m$.
\begin{lemma} \label{lem:ProperEvent} 
Let $\beta >1$ and $0 < \zeta /\alpha  < 2$. Assume that $m\geq 2$ and $k_1, \ldots, k_m \geq 0$ are integers. 
Then we have 
\begin{equation*}
\Prb {\A (k_1,\ldots, k_m) \ | \ \E_1 \cap \mathcal{T}_1} = (1+o(1)) \prod_{i=1}^m \Prb {D_{v_i} = k_i}.
\end{equation*}
\end{lemma}
\begin{proof}
Note that if the positions of $v_1,\ldots, v_m$ have been fixed, then $\{  \cup_{i=1}^m \A_{v_i}^w \}_{w \in V_N^{v_1,\ldots, v_m}}$ is 
an independent family. 
Thus, assuming that the positions $(t_{v_1}, \theta_{v_1}),\ldots, (t_{v_m},\theta_{v_m})$ of $v_1,\ldots, v_m$ in $\D$ have been exposed 
so that $\E_1 \cap \mathcal{T}_1$ is realized, we can write
\begin{equation} \label{eq:EssentialProb}
\begin{split}
\Prb {\A (k_1,\ldots, k_m) & \ | \ (t_{v_1}, \theta_{v_1}),\ldots, (t_{v_m},\theta_{v_m})  } =  \\
&{N -m \choose k_1 k_2 \cdots N- \sum_{i=1}^m k_i} \left[ \prod_{i=1}^m \Prb {\A_{v_i} }^{k_i} \right]
 \left(1- \sum_{i=1}^m \Prb {\A_{v_i} }\right)^{N-\sum_{i=1}^m k_i}.
 \end{split}
\end{equation}
We now proceed by giving an estimate for $\Prb {\A_{v_i} }$. That is, we will calculate the probability that a 
vertex $w \not \in \{ v_1,\ldots, v_m\}$ is located within $A_{v_i}$ and it is adjacent to $v_i$. 
Setting $x_0' = \min \{ x_0,  R-t_{v_i} - \omega (N)\}$, we have 
\begin{equation} \label{eq:ProbA}
\begin{split}
\Prb {\A_{v_i}} &= {1\over \pi} \int_{0}^{x_0} \int_0^{\hat{\theta}_{v_i,w}'} p_{v_i,w} \rho (t_w) d \theta  dt_w \\ 
& = {1\over \pi} \int_{0}^{x_0'} \int_0^{\hat{\theta}_{v_i,w}'} p_{v_i,w} \rho (t_w) d \theta  dt_w 
+ {1\over \pi} \int_{x_0'}^{x_0} \int_0^{\hat{\theta}_{v_i,w}'} p_{v_i,w} \rho (t_w) d \theta  dt_w.
\end{split}
\end{equation} 
The second integral is bounded as in (\ref{eq:IndExp2nd}). 
In particular, it is bounded from above by  
\begin{equation*}
{1\over \pi} \int_{R-t_{v_i} - \omega (N)}^R \int_0^{\hat{\theta}_{v_i,w}'} p_{v_i,w} \rho (t_w) d \theta  dt_w 
= O \left(e^{\alpha \omega (N)}\left( {e^{\zeta t_{v_i}/2} \over N} \right)^{2\alpha /\zeta} \right).
\end{equation*}
Regarding the first integral, we argue as in (\ref{eq:1stIntegral}), (\ref{eq:2ndIntegral}) and (\ref{eq:IntLower}). 
Recall that for $\beta > 1$, we defined $\tilde{\theta}_{v_i,w} = {1\over \omega (N)}~A_{v_i,w}^{-1}$. 
\begin{equation*} 
\begin{split} 
\int_0^{\hat{\theta}_{v_i,w}'} p_{v_i,w} d \theta = \int_0^{\tilde{\theta}_{v_i,w}} p_{v_i,w} d \theta + 
\int_{\tilde{\theta}_{v_i,w}}^{\hat{\theta}_{v_i,w}'} p_{v_i,w} d \theta 
= \int_{\tilde{\theta}_{v_i,w}}^{\hat{\theta}_{v_i,w}'} p_{v_i,w} d \theta + o\left( A_{v_i,w}^{-1}\right). 
\end{split}
\end{equation*} 
For the first integral, we imitate the calculation in (\ref{eq:IntLower}), expressing $p_{v_i,w}$ using Lemma~\ref{lem:Dist} 
and applying the transformation  $z=C^{1/\beta}~A_{v_i,w}~{\theta \over 2}$. We obtain
\begin{equation*} 
\int_{\tilde{\theta}_{v_i,w}}^{\hat{\theta}_{v_i,w}'} p_{v_i,w} d \theta  = (1+o(1))~{C_{\beta} \over A_{v_i,w}},
\end{equation*}
where $C_\beta$ is as in Lemma~\ref{lem:AngleAv} for $\beta > 1$. 

Thereby, as in the previous section, the first integral in (\ref{eq:ProbA}) becomes
\begin{equation*}
\begin{split} 
&{1\over \pi} \int_{0}^{x_0'} \int_0^{\hat{\theta}_{v_i,w}'} p_{v_i,w} \rho (t_w) d \theta  dt_w  
= (1+o(1)) C_\beta \int_{0}^{x_0'} A_{v_i,w}^{-1} \rho (t_w) d t_w \\
& (1+o(1)) {2\alpha  C_\beta \over 2\alpha  - \zeta} {e^{\zeta t_{v_i}/2} \over N}. 
\end{split}
\end{equation*}
With $K=2 \alpha C_\beta/ (2\alpha - \zeta)$, as it was set in Lemma~\ref{lem:IndExp} for $\beta > 1$ and
substituting the above estimates into (\ref{eq:ProbA}) we obtain
\begin{equation} \label{eq:ProbAFinal} 
\begin{split}
\Prb {\A_{v_i}}  = (1+o(1))~K~{e^{\zeta t_{v_i}/2} \over N}.
\end{split}
\end{equation}
Under the assumption that $t_{v_i} \leq R/2 - \omega (N)$, we have $e^{\zeta t_{v_i}/2} / N = o\left( {1\over N^{1/2}}\right)$.
Thus, if $t_{v_i} \leq R/2 - \omega (N)$, for all $i=1,\ldots, m$ 
\begin{equation*}
\begin{split}
 \left(1- \sum_{i=1}^m \Prb {\A_{v_i} }\right)^{N-\sum_{i=1}^m k_i} = \exp \left( - (1+o(1))~K 
\sum_{i=1}^m  e^{\zeta t_{v_i}/2} \right).
\end{split}
\end{equation*}
Substituting this estimate as well as that in (\ref{eq:ProbAFinal}) into (\ref{eq:EssentialProb}) we obtain that uniformly for
all $(t_{v_1},\ldots , t_{v_m}) \in [0, \min \{R/2 - \omega (N) , x_0 \}]^{m}$ and all $\theta_{v_1},\ldots, \theta_{v_m} \in (0,2\pi ]$ 
such that $\E_1 \cap \mathcal{T}_1$ is realized:
\begin{equation} \label{eq:PoissonApx}
\begin{split}
& \Prb  {\A (k_1,\ldots, k_m) \ | \ (t_{v_1}, \theta_{v_1}),\ldots, (t_{v_m},\theta_{v_m}) } = \\
& (1+o(1))  
\prod_{i=1}^m {\left( K e^{\zeta t_{v_i}/2} \right)^{k_i}\over k_i!} \exp \left(-(1+o(1)) K e^{\zeta t_{v_i}/2} \right) 
 = (1+o(1)) \prod_{i=1}^m \Prb {D_{v_i} = k_i},
\end{split}
\end{equation}
by Lemma~\ref{lem:PoissonApx}. 
By (\ref{eq:2ndInt}), the probability that there exists an index $i$ with $1\leq i \leq m$ such that $t_{v_i} > R/2 - \omega (N)$ is $o(1)$. 
Hence, averaging over all $(t_{v_i}, \theta_{v_i})$, for $i=1,\ldots, m$, on the measure conditional on $\E_1 \cap \mathcal{T}_1$, 
the lemma follows. 
\end{proof}
We conclude the proof of (\ref{eq:2ndThmToProve}) with the proof of (\ref{eq:B1}). 
\begin{lemma} \label{lem:B1Apx}
For any $\beta > 1$ and any $0 < \zeta/\alpha  < 2$ we have 
$$\Prb{\mathcal{B}_1} = o(1).$$ 
\end{lemma}
\begin{proof}
For a given $w \in V_{N}^{v_1,\ldots, v_m}$ and $i \in [m]$, 
the probability of the event $\tilde{A}_{v_i}^w$, conditional on $t_{v_i}$, can be written as follows: 
\begin{equation} \label{eq:B1Prob}
\begin{split}
\Prb {\tilde{\A}_{v_i}^w} &= {1\over \pi}\int_0^R \int_{\hat{\theta}_{v_i,w}'}^{\pi} p_{v_i,w} \rho (t_w) d \theta d t_w \\
&=  {1\over \pi}\int_0^{R  - t_{v_i}-\omega (N)} \int_{\hat{\theta}_{v_i,w}'}^{\pi} p_{v_i,w} \rho (t_w) d \theta d t_w
+ {1\over \pi}\int_{R-t_{v_i}-\omega (N)}^R \int_{\hat{\theta}_{v_i,w}'}^{\pi} p_{v_i,w} \rho (t_w) d \theta d t_w. 
\end{split}
\end{equation}
The second integral can be bounded as in (\ref{eq:IndExp2nd}) - we have 
\begin{equation} \label{eq:B1_2ndIntegral} 
{1\over \pi}\int_{R-t_{v_i}-\omega (N)}^R \int_{\hat{\theta}_{v_i,w}'}^{\pi} p_{v_i,w} \rho (t_w) d \theta d t_w 
= O \left( e^{\alpha \omega (N)}~\left( {e^{\zeta t_{v_i}/2}\over N}\right)^{2\alpha /\zeta} \right). 
\end{equation}
Regarding the first integral, we use the estimate obtained in Lemma~\ref{lem:Dist} to bound the inner integral. 
With $C$ as in the proof of Lemma~\ref{lem:AngleAv}, we have 
\begin{equation*} 
\begin{split}
\int_{\hat{\theta}_{v_i,w}'}^{\pi} p_{v_i,w} d \theta &= \int_{\hat{\theta}_{v_i,w}'}^{\pi} 
{1\over C A_{v_i,w}^{\beta} \sin^{\beta} \left( {\theta \over 2} \right) +1} d \theta \leq 
{1\over C A_{v_i,w}^{\beta}}~\int_{\hat{\theta}_{v_i,w}'}^{\pi} {1 \over \sin^{\beta} \left( {\theta \over 2} \right)} d\theta \\
& \stackrel{\sin \left( {\theta \over 2} \right) \geq {\theta \over \pi}}{\leq}  {\pi^{\beta} \over C A_{v_i,w}^{\beta}}~
\int_{\hat{\theta}_{v_i,w}'}^{\pi} \theta^{-\beta} d\theta = 
{\pi^{\beta} \over C (\beta -1) A_{v_i,w}^{\beta}}~\left( \hat{\theta}_{v_i,w}'^{-\beta + 1} - \pi^{-\beta +1}\right) \\
& = O\left( {A_{v_i,w}^{-1} \over \omega (N)^{\beta -1}}\right).
\end{split}
\end{equation*}
Thus, the first integral in (\ref{eq:B1Prob}) becomes
\begin{equation}  \label{eq:B1_1stIntegral}
\begin{split} 
&{1\over \pi}\int_0^{R  - t_{v_i}-\omega (N)} \int_{\hat{\theta}_{v_i,w}'}^{\pi} p_{v_i,w} \rho (t_w) d \theta d t_w = 
O \left( {1\over \omega (N)^{\beta -1}}~{e^{\zeta t_{v_i}/2} \over N}\right)~
\int_0^{R  - t_{v_i}-\omega (N)} e^{\zeta t_w /2} \rho (t_w) d t_w \\
& \stackrel{(\ref{eq:ColdIndExp1st})}{=} O \left( {1\over \omega (N)^{\beta -1}}~{e^{\zeta t_{v_i}/2} \over N}\right).
\end{split}
\end{equation}
Now, we take the average of each one of the bounds obtained in (\ref{eq:B1_1stIntegral}) and (\ref{eq:B1_2ndIntegral}), respectively, 
over $t_{v_i}$. To this end, we need the following integral, whose simple calculation we omit.
 We have  
\begin{equation*} \label{eq:BasicIntegral} 
\int_{0}^R e^{at} \rho(t) dt = \begin{cases}
\Theta (R) & \mbox{if $a=\alpha$} \\
\Theta (1) & \mbox{if $a = \zeta /2$}.
 \end{cases}
\end{equation*}
Thus, the bound in (\ref{eq:B1_2ndIntegral}) is $O\left( e^{\omega (N)}~{R \over N^{2/\zeta}}\right)$ and that in 
(\ref{eq:B1_1stIntegral}) is $O \left( {1\over \omega (N)^{\beta -1} N }\right)$. Since $\zeta/\alpha < 2$, both terms are $o(N^{-1})$. 
Therefore, the union bound implies that 
\begin{equation} \label{eq:Outside}
\Prb {\cup_{w \in V_{N}^{v_1,\ldots , v_m}} \cup_{i=1}^m \tilde{\A}_{v_i}^w} = o(1). 
\end{equation}
\end{proof}
Thus, the estimates obtained in Lemmas~\ref{lem:ProperEvent} and~\ref{lem:B1Apx} substituted in (\ref{eq:ToDeduce}) 
imply (\ref{eq:2ndThmToProve}). 


\section{Conclusions - Open Questions}

This paper initiates a rigorous study of random geometric graphs on spaces of negative curvature. We considered the 
binomial model where the vertices are points that are randomly placed on a hyperbolic space and each pair is included as 
an edge of the random graph with probability that depends on their hyperbolic distance, independently of every other pair. 
This probability also depends on a parameter $\beta > 0$, which turns out to determine the typical ``behaviour" of the resulting 
random graph. We establish $\beta = 1$ as the critical value around which a transition occurs. Namely, when $\beta > 1$, the 
random graph is sparse and exhibits power-law degree sequence, whereas for $\beta < 1$, the degree of a typical vertex grows 
polynomially with $N$, which is the total number of vertices of the graph. For $\beta = 1$, the degree of a typical vertex grows logarithmically 
in $N$. 

This study raises a number of questions regarding the typical structure of these random graphs for various values of $\beta$ as well as 
the transition itself when $\beta$ ``crosses" the critical value $\beta = 1$. For example, what is the diameter of the random graph 
and the typical distance between two vertices that belong to the same component? Is there a giant component and, if yes, 
what is the distribution of the smaller components. What is the clustering coefficient of such a random graph and how does it depend on
$\beta$? When $\beta \leq 1$, is the random graph connected, and if yes, is it Hamiltonian? 
Can one describe the evolution of the random graph as $\beta$ approaches 1 from above?

\medskip

{\footnotesize \obeylines \parindent=0pt

Nikolaos Fountoulakis
School of Mathematics
University of Birmingham
Edgbaston
Birmingham
B15 2TT
UK
}
\begin{flushleft}
{\it{E-mail address}:
\tt{{n.fountoulakis@bham.ac.uk}}}
\end{flushleft}

\end{document}